\documentclass[smallextended]{svjour3}

\usepackage{graphicx}
\usepackage{amssymb}
\usepackage{amsmath}

\journalname{Numerical Algorithms}

\begin{document}

\title{Fitted Finite Volume Method for a Generalized Black-Scholes Equation Transformed on Finite Interval}

\author{Radoslav Valkov}

\institute{Radoslav Valkov \at Faculty of Mathematics and Informatics, University of Sofia, \\ 1000 Sofia, Bulgaria, \\ \email{rvalkov@fmi.uni-sofia.bg}}

\date{Received: date / Accepted: date}

\maketitle

\begin{abstract}
A generalized Black-Scholes equation is considered on the semi-axis. It is transformed on the interval $(0,1)$ in order to make the computational domain finite. The new parabolic operator degenerates at the both ends of the interval and we are forced to use the G\"{a}rding inequality rather than the classical coercivity. A fitted finite volume element space approximation is constructed. It is proved that the time $\theta$-weighted full discretization is uniquely solvable and positivity-preserving. Numerical experiments, performed to illustrate the usefulness of the method, are presented.
\end{abstract}


\section{Introduction}

The famous equation, proposed by F. Black, M. Scholes and R. Merton, see \cite{Duffy,Sevc,WHD}, is

\begin{equation*}
\frac{\partial V}{\partial t} + \frac{1}{2} \sigma^{2} S^{2} \frac{\partial^{2} V}{\partial S^{2}} + r \left(\frac{\partial V}{\partial S } - SV \right)=0, \; S \in (0,\infty), \; t \in [0,T).
\end{equation*}

This is a typical example of a \emph{degenerate} parabolic equation \cite{OR}. It is well known \cite{Sevc,WHD,Zhu} that it can be transformed to the heat equation that allows us to overcome the degeneracy at $S=0$. Many numerical methods, based on classical finite difference schemes, applied to constant coefficients heat equations, are developed \cite{AP,RS}. However, when the problem has space-dependant coefficients $\sigma$ and $r$ one can not transform the Black-Scholes equation to a standard heat equation. Finite difference and finite element methods have been applied in \cite{A,CL1,CL2,CS,CV,Ehrhardt1,Ehrhardt2,Sulli,Vasq,Zhu} in order to solve this type of generalized Black-Scholes equations. In \cite{Kadalbajoo} cubic $B$-splines are implemented. Often, the convergence of the full discretization is verified by numerical examples only.

An effective method, that resolves the \emph{degeneracy}, is proposed by S. Wang \cite{SW} for the Black-Scholes equation with Dirichlet boundary conditions. The method is based on a finite volume formulation of the problem, coupled with a fitted local approximation to the solution and an implicit time-stepping technique. The local approximation is determined by a set of two-point boundary value problems (BVPs), defined on the element edges. This fitted technique originates from one-dimensional computational fluid dynamics \cite{Morton}.

A modification of the discretization, originally presented in \cite{SW}, was proposed by L. Angermann \cite{A} such that the method adequately treats the proper (natural) boundary condition at $x=0$. Similar space discretization is derived in \cite{CV} for a degenerate parabolic equation in the zero-coupon bond pricing.

The domain of $S$ is the half real line. For numerical computation it is desirable to have a \emph{finite computational domain}. The transformation in the next section converts $S \in (0,\infty)$ to $x \in (0,1)$, \emph{decreasing significantly the computational costs}. Also, for a call option, the solution $V(S,t)$ is not bounded and from the numerical methods' point of view the problem transformation on a finite interval is better. The resulting equation has variable coefficients but this is not an essential difficulty for the numerical computation. However, the transformed equation \emph{degenerates} at both ends of the finite interval.

The present paper deals with a degenerate parabolic equation, \eqref{TransfEq}, derived after transformation of the generalized Black-Scholes equation \eqref{OrigEq} to a finite interval. The degeneration at the both ends of the interval does not allow the use of the Poincar\'{e}-Friedrichs inequality and we are forced to investigate the differential problem with the G\"{a}rding inequality rather than classical coercivity \cite{GV}.

This paper is organized as follows. The model problem is presented in Section 2, where we discuss our basic assumptions and some properties of the solution. The space discretization method is developed in Section 3. Section 4 is devoted to the time discretization, where we show that the system matrix on each time-level is an $M$-matrix so that the discretization is monotone. Numerical experiments are discussed in the last section.

Some results, concerning the case of the transformed Black-Scholes equation above, are reported in \cite{RV}.

\section{The transformed problem}

We consider the generalized Black-Scholes equation \cite{Sevc,WHD}:

\begin{equation}\label{OrigEq}
\frac{\partial V}{\partial t} + \frac{1}{2}\sigma^{2}(t) S^{2}\frac{\partial^{2}V}{\partial S^{2}} + (r(t)S - D(S,t))
\frac{\partial V}{\partial S} - r(t)V = 0, \; (S,t)  \in (0,\infty ) \times (0,T),
\end{equation}
where $\sigma = \sigma (t)$ denotes the \emph{volatility} of the asset, $r=r(t)$ is the \emph{risk-free interest rate}, $D=D(S,t)$ denotes the \emph{dividend} of the dividend-paying asset. We also introduce $d=d(S,t)$ such that $D(S,t)= S d(S,t)$, where \emph{the dividend rate} $d=d(S,t)$ is continuously differentiable with respect to $S$.

There are various choices for the final (\emph{payoff}) condition, depending on the models. In the case of vanilla European option

\begin{equation}\label{FinalC1}
V(S,T)=\left\{
       \begin{array}{ll}
         \max (S-E,0) & \mbox{for a call option}, \\
         \max (E-S,0) & \mbox{for a put option},
       \end{array}
     \right.
\end{equation}
where $E$ is the strike price. Another example is the \emph{bullish vertical spread} payoff, defined by
\begin{equation}\label{FinalC2}
V(S,T)= max(S-E_1,0)-max(S-E_2,0),
\end{equation}
where $E_1$ and $E_2$ are two exercise prices, satisfying $0<E_1<E_2$. This represents a portfolio of buying one call option with exercise price $E_1$ and issuing one call option with the same expiry date but a larger exercise price $E_2$. For detailed discussion on this, we refer to \cite{WHD}.

We introduce the transformation \cite{Zhu}

\begin{equation}\label{ZhuTransform}
x = \frac{S}{S+P_{m}}, \; u(x,t) = \frac{V(S,t)}{S+P_{m}}, \; \tau = T-t.
\end{equation}
The constant $P_m$ is called \emph{mesh parameter}. It controls the distribution of the mesh nodes w.r.t. $S$ on the interval $(0,\infty)$. The higher the value of S, that we are interested in, the higher value of $P_m$ should be in order to obtain a reasonable accuracy. In the case of a call option, because of the nature of the terminal condition, $P_m$ should be equal to $E$.

The inverse transformation is $S=P_{m}x/(1-x)$ and after plugging it in the Black-Scholes equation, \eqref{OrigEq}, we obtain:

\begin{eqnarray}\label{TransfEq}
\begin{split}
& \frac{\partial u}{\partial t} - \frac{1}{2} \sigma^{2}(t) x^{2}(1-x)^{2} \frac{\partial^{2} u}{\partial x^{2}}
  - x(1-x)(r(t) - d(x,t))\frac{\partial u}{\partial x} \\
& + ((1-x)r(t) + x d(x,t))u = 0, \; (x,t) \in \Omega_{T} = \Omega \times (0,T), \Omega=(0,1).
\end{split}
\end{eqnarray}

We return to the original notation of the variable $t$ for the sake of simplicity. The initial data for a call option reads

\begin{equation}\label{InitialTransf}
u(x,0) = u_{0}(x) = \max (2x-1,0).
\end{equation}

Being different from the classical parabolic equations, in which the principle coefficient is assumed to be strictly positive, the parabolic equation \eqref{TransfEq} belongs to the second-order differential equations with non-negative characteristic form \cite{OR}. The main difficulty of such kind of equations is the \emph{degeneracy} \cite{WFV}. It can be easily seen that at $x=0$ and $x=1$ \eqref{TransfEq} degenerates to

\begin{equation}\label{NaturalBC}
\left. \frac{\partial u}{\partial t} \right\vert_{x=0}= -r(t)u(0,t), \; \left. \frac{\partial u}{\partial t} \right\vert_{x=1} = -d(1,t)u(1,t).
\end{equation}

It is well known by the Fichera theory for degenerate parabolic equations \cite{OR} that at the degenerate boundaries $x=0$ and $x=1$ the boundary conditions should not be given.

For theoretical analysis of our discrete problem as well as for the construction of a fitted finite volume mass-lumping discretization we need to consider weak solutions of \eqref{TransfEq}. We shall use the standard notations for the function spaces $C^{m}(\Omega)$ and $C^{m}({\overline{\Omega}})$ of which a function and it's derivatives up to order $m$ are continuous on $\Omega$ (respectively ${\overline{\Omega}}$). The space of square-integrable functions we denote by $L^{2}(\Omega)$ with the usual $L^{2}$-norm $\Vert \cdot \Vert$ and the inner product $(\cdot,\cdot)$. We also use the function space $L^{\infty}(\Omega)$ with the norm $\Vert \cdot \Vert_{\infty}$. To handle the degeneracy in \eqref{TransfEq}, we introduce the following weighted $L_{w}^{2}$-norm
\begin{equation*}
\Vert v \Vert_{0,w} := (\int_{0}^{1} x^{2}(1-x)^{2} v^{2} dx )^{1/2}
\end{equation*}
with corresponding inner product $(u,v)_{w}$. Using $L^{2}(\Omega)$ and $L_{w}^{2}(\Omega)$, we define the weighted Sobolev space
$H_{w}^{1}(\Omega) := \{v : v \in L^{2}(\Omega), \; v' \in L_{w}^{2}(\Omega)\}$, where $v'$ denotes the weak derivative of $v$. Let $\Vert \cdot \Vert_{1,w}$ be the functional on $H_{w}^{1}(\Omega)$, defined by $\Vert v \Vert_{1,w} = (\Vert v \Vert_{0}^{2} + \Vert v' \Vert_{0,w}^{2})^{1/2}$. Then it is easy to see that
$\Vert \cdot \Vert_{1,w}$ is a norm on $H^{1}(\Omega)$; it is called weighted $H^{1}$-norm on $H_{w}^{1}(\Omega)$. Furthermore, using the inner products in
$L^{2}(\Omega)$ and $L_{w}^{2}(\Omega)$, we define a weighted inner product on $H_{w}^{1}(\Omega)$ by
$(\cdot , \cdot )_{H} := (\cdot , \cdot) + (\cdot , \cdot )_{w}$ and, consequently, the pair $(H_{w}^{1}(\Omega),(\cdot , \cdot)_{H})$ is a Hilbert space. Also, $H_{w}^{1}(\Omega)$ contains the conventional Sobolev space $H^{1} (\Omega)$ as a proper subspace.

We rewrite \eqref{TransfEq} in divergent form

\begin{eqnarray}\label{DivEq}
\begin{split}
& \frac{\partial u}{\partial t} - \frac{\partial }{\partial x}\left(x(1-x) \left(a(x,t)\frac{\partial u}{\partial x} + b(x,t)u \right)\right) + c(x,t)u = 0, \\
& a(x,t) = \frac{1}{2}\sigma^{2} (t) x(1-x), \; b(x,t) = r(t) - d(x,t) + \sigma^{2}(t) (2x-1), \\
& c(x,t) = (2-3x) r(t) - (6x^{2} - 6x + 1)\sigma^{2}(t) - (1-3x)d(x,t)- x(1-x) \frac{\partial d}{\partial x}.
\end{split}
\end{eqnarray}

Let us introduce for $w,v \in H_{w}^{1}(\Omega)$ the bilinear form
\begin{equation*}
  A(w,v;t) := (aw'+ bw, v') + (cw,v) = (x(1-x) \rho (w) , v') + (cw,v).
\end{equation*}

Here the notation $w' = \frac{\partial w}{\partial x}$ is used and the function $\rho (w)=aw' + bw$ is the \emph{weighted flux density}, associated with $w$.
We are in position to state the variational formulation of problem \eqref{TransfEq},\eqref{InitialTransf}:

\emph{Find $u(t)\in H_{w}^{1}(\Omega)$, such that for all $v \in H_{w}^{1}(\Omega)$}
\begin{equation} \label{VarForm}
\left( \frac{\partial u(t) }{\partial t}, v \right) + A(u(t),v;t) =0 \; \mbox{a.e. in} \; (0,T) \; \mbox{and} \; u(\cdot,0)=u_{0}.
\end{equation}

The following result provides the \emph{weak coercivity and continuity} of the bilinear form $A(\cdot,\cdot,t)$.

\begin{lemma}\label{Garding} Let $w,v \in C^{1} ([0,1])$. Then:
\begin{enumerate}
  \item
    there exist constants $C_{1} > 0$ and $C_{2} > 0$ such that
    \begin{eqnarray*}
    \vert A(w,v,t) \vert \leq C_{1} \Vert w \Vert_{H_{w}^{1}(0,1)}\Vert v \Vert_{H^{1}(0,1)} \label{BoundEst1} \\
    \vert A(w,v,t) \vert \leq C_{2} \Vert w \Vert_{H^{1}(0,1)}\Vert v \Vert_{H^{1}_{w}(0,1)} \label{BoundEst2}
    \end{eqnarray*}
    for any $ t \in [0,T]$;
  \item
    \textbf{(G\"{a}rding inequality)} there exist constants $\alpha > 0$ and $\gamma > 0$ such that
    \begin{equation*}
    \vert A(v,v,t) \vert \geq \alpha \Vert v \Vert_{H_{w}^{1}(0,1)}^{2} - \gamma \Vert v \Vert_{L^{2}(0,1)}^{2},
    \end{equation*}
\end{enumerate}
uniformly with respect to $t \in [0,T]$.
\end{lemma}

\begin{proof} The proof is given in \cite{GV}.
\end{proof}

Owing to Lemma \ref{Garding} one can prove the following assertion \cite{GV}.

\begin{theorem}
Suppose that $u_0(x) \in H_{w}^{1}(\Omega)$. Then the problem \eqref{VarForm} has an unique solution.
\end{theorem}

\begin{theorem}\label{WMP}
(\textbf{Weak maximum principle}) Let $u\in L^{2}\left(0,T;H_{w}^{1}(\Omega)\right)$ be such that
\begin{enumerate}
\item $\frac{\partial u}{\partial t}\in L^{2}\left(\left(0,1\right) \times \left(0,T\right) \right)$;
\item the inequality
\begin{equation*}
\left( \frac{\partial u(t) }{\partial t}, v \right) + A(u(t),v;t) \geq 0, \qquad \forall v \in C_{0}^{\infty} \left(0,1\right), v \geq 0
\end{equation*}
holds for a.e. $\tau \in \left[0,T\right]$;
\item $\left. u\right\vert _{t=0} \geq 0$.
\end{enumerate}
Then $u \geq 0$ a.e. in $\left(0,1\right) \times \left(0,T\right)$.
\end{theorem}

\begin{proof} The proof is given in \cite{GV}.
\end{proof}

\section{Space discretization}

In this section we describe the finite volume approximation of \eqref{DivEq}.

Let the interval $\Omega = [0,1]$ be subdivided into $N$ intervals $I_{i} := (x_{i},x_{i+1}), \; i=0,\dots,N-1,$ with $0 =: x_{0} < x_{1} < \dots < x_{N}:=1$. For each $i=0,\dots,N-1$ we set $h_{i} := x_{i+1} - x_{i}$ and $h := \max_{i=0,\dots,N-1} h_{i}$. We also denote
$x_{i+1/2} := x_{i}+h_{i}/2$ for $i=0,\dots,N-1$, $x_{-1/2}:=x_{0}=0$, $x_{N+1/2}:=x_{N}=1$ and $\Omega_{i}:= [x_{i-1/2},x_{i+1/2}]$ for $ i=0,\dots,N$.
Finally, we define $l_{i} := x_{i+1/2}-x_{i-1/2}$ for $i=0,\dots,N$.

Integrating \eqref{DivEq} over the control volumes $\Omega_{i}$ we obtain $N+1$ balance equations
\begin{equation*}
\int_{\Omega_{i}}{\dot u}dx - \left[x(1-x) \rho (u)\right]_{x_{i-1/2}}^{x_{i+1/2}} + \int_{\Omega_{i}} cu dx =0, \; \; i=0,\dots,N.
\end{equation*}

Multiplying the i-th equation with an arbitrary number $v_{i}$ and summing up the results we get

\begin{equation}\label{FVEq}
\sum_{i=0}^{N} \int_{\Omega_{i}}{\dot u} v_{i} dx - \sum_{i=0}^{N}\left[x(1-x) \rho(u) \right]_{x_{i-1/2}}^{x_{i+1/2}}v_{i} + \sum_{i=0}^{N}
\int_{ \Omega_{i}} cu v_{i} dx = 0.
\end{equation}

For an arbitrary function $v \in C({\overline {\Omega}})$ we define the mass-lumping operator $L : C( {\overline {\Omega}}) \to L_{\infty}(\Omega)$ by
$L_{h} v \vert_{ \Omega_{i} } := v (x_{i}), \; i=0,\dots,N$.

Therefore, using the operator $L_{h}$, equation \eqref{FVEq} can be written as follows:

\begin{eqnarray*}
& \left( {\dot u}(t), L_{h} v \right) + A_h(u,v;t) = 0, \; \forall v \in C({\overline{\Omega}}), \\
& A_h (w,v;t) := - \sum_{i=0}^{N} \left[x(1-x) \rho (w,x,t) \right]_{x_{i-1/2}}^{x_{i+1/2}} \left.L_{h}v\right|_{\Omega_i} + (c(t)w,L_{h}v).
\end{eqnarray*}

The spatial discretization starts from this equation. Applying the mid-point quadrature rule to all terms except the second one we obtain for all $v \in C({\overline{\Omega}})$

\begin{eqnarray*}
& ({\dot u}(t),v)_{h} - \sum_{i=0}^{N}\left[x(1-x) \rho (u(t),x,t)\right]_{x_{i-1/2}}^{x_{i+1/2}} \left.L_{h}v\right|_{\Omega_i} + (c(t) u,v)_{h} = 0, \\
& (w,v)_{h} := (L_{h}w,L_{h}v) = \sum_{i=0}^{N} w_{i} v_{i} l_{i}, \; w,v \in C({\overline{\Omega}}).
\end{eqnarray*}

Clearly, we now need to derive approximations of the continuous weighted flux density $x(1-x)\rho(u(t),x,t)$, defined above, at the midpoints $x_{i+1/2}$ of the intervals $I_i$ for $i=0,\dots,N-1$.

\textbf{Case 1} Approximation of $\rho$ at $x_{i+1/2}$ for $1 \leq i \leq N-2$.

Let us consider the following two-point boundary value problem for $x \in I_{i}$
\begin{equation*}
(a_{i+1/2}x(1-x)v' + b_{i+1/2}v)' =0; \; \; v(x_{i}) = u_{i}, v(x_{i+1}) = u_{i+1},
\end{equation*}
where $a_{i+1/2} = a(x_{i+1/2}), \; \; b_{i+1/2} = b(x_{i+1/2},t)$.

Following considerations, similar to those in \cite{CV,RV}, we obtain

\begin{equation}\label{FluxMid}
\rho_{i}(u) = b_{i+1/2} \frac{\left( \frac{x_{i+1}}{1-x_{i+1}} \right)^{\alpha_{i}} u_{i+1} - \left( \frac{x_{i}}{1-x_{i}} \right)^{\alpha_{i}} u_{i}}
{\left(\frac{x_{i+1}}{1-x_{i+1}}\right)^{\alpha_{i}} - \left(\frac{x_{i}}{1-x_{i}}\right)^{\alpha_{i}}}, \; \;i=1,\dots,N-2,
\end{equation}
where $\alpha_{i} = \frac{b_{i+1/2}}{a_{i+1/2}}$ and $\rho_{i}(u)$ provides an approximation to $\rho(u)$ at $x_{i+\frac{1}{2}}$.

\textbf{Case 2} Approximation of $\rho$ at $x_{1/2}$.

Now we write the flux in the form
\begin{equation*}
\rho (u) := {\overline{a}} x \frac{\partial u}{\partial x} + bu, \; \; \overline{a}=\overline{a}(x)=\frac{\sigma^{2}}{2}(1-x).
\end{equation*}

Note that the analysis in Case 1 can not be applied here because the flux degenerates at $x=0$. To solve this difficulty, following \cite{A,CV,RV,SW}, we
will reconsider the flux ODE with an extra degree of freedom in the following form

\begin{equation*}
({\overline{a}}_{1/2} x v' + b_{1/2} v)' = C \; \; \mbox{in} \; \; (0, x_{1}), \; \; v(0)=u_{0} , \; \; v(x_{1}) =u_{1},
\end{equation*}
where $C$ is an unknown constant to be determined. We obtain

\begin{equation*}
v(x)=\left\{
       \begin{array}{ll}
         x \frac{u_{1}-u_{0}}{h_{0}}+u_{0}, & \alpha_{0} \geq 0, \\
         \left(\frac{x}{h_{0}}\right)^{-(1+\alpha_{0})}x \frac{u_{1}-u_{0}}{h_{0}}+u_{0}, & \alpha_{0} < 0,
       \end{array}
     \right.
\end{equation*}

\begin{equation}\label{FluxLeft}
\rho_{0} (v):= \left\{
       \begin{array}{ll}
         \overline{a}_{1/2} (1+\alpha_{0}) x \frac{u_{1}-u_{0}}{h_{0}}+b_{1/2}u_{0}, & \alpha_{0} \geq 0, \\
         b_{1/2}u_{0}, & \alpha_{0} < 0.
       \end{array}
     \right.
\end{equation}

\textbf{Case 3} Approximation of $\rho$ at $x_{N-1/2}$.

We write the flux in the form

\begin{equation*}
\rho (u) := \stackrel{=}{a}_{N-1/2} (1-x) \frac{\partial u}{\partial x} + b_{N-1/2}u, \; \stackrel{=}{a}(x) = \frac{\sigma^{2}}{2}x.
\end{equation*}
The situation is symmetric to Case 2. We can not apply the arguments in Case 1 to the approximation of the weighted flux density on $I_{N-1} = (x_{N-1},x_{N})$ because equation \eqref{TransfEq} degenerates at $x=x_{N}=1$. However the considerations, given in Case 2, should be modified in order to formulate an appropriate two-point BVP. Again, we consider the flux ODE with an extra degree of freedom in the following form (recall $\alpha_{N-1}=b_{N-1/2}/{\stackrel{=}{a}}_{N-1/2}$)

\begin{equation}\label{FluxODERight1}
((1-x)v' + \alpha_{N-1}v)' = C_{0}, \; x \in I_{N-1},
\end{equation}
\begin{equation}\label{FluxODERightBC}
v(x_{N-1}) = u_{N-1}, \; v(x_{N}) = u_{N},
\end{equation}
where $C_{0}$ is an unknown constant to be determined. Integration of \eqref{FluxODERight1} yields the first-order linear equation

\begin{equation}\label{FluxODERight2}
(1-x)v' + \alpha_{N-1}v = C_{0}x + C_{1}, \; x \in I_{N-1},
\end{equation}
where $C_{1} $ denotes an additive constant. Afterwards we multiply \eqref{FluxODERight2} by $(1-x)^{-\alpha_{N-1}-1}$

\begin{equation}\label{FluxODERight3}
\left( (1-x)^{-\alpha_{N-1}}v \right)^{'}=C_{0}(1-x)^{-\alpha_{N-1}-1}x + C_{1}(1-x)^{-\alpha_{N-1}-1}.
\end{equation}

\textbf{Case 3.1} $ \alpha_{N-1} > 0, \alpha_{N-1} \neq 1 $.

Integrating \eqref{FluxODERight3} from $ x_{N-1} $ to $x \in I_{N-1} $ results in

\begin{eqnarray*}
& (1-x)^{-\alpha_{N-1}}v(x) - (1-x_{N-1})^{-\alpha_{N-1}}v(x_{N-1}) \\
& =C_{0}\left. \frac{(1-s)^{-\alpha_{N-1}+1}}{-\alpha_{N-1}+1} \right\vert_{x_{N-1}}^{x}-( C_{0} + C_{1} ) \left. \frac{(1-s)^{-\alpha_{N-1}}}{-\alpha_{N-1}}\right\vert_{x_{N-1}}^{x}.
\end{eqnarray*}

Multiplying both sides of the equation by $(1-x)^{\alpha_{N-1}}$

\begin{eqnarray*}
& v(x)=\frac{(1-x)^{\alpha_{N-1}}}{(1-x_{N-1})^{\alpha_{N-1}}}v(x_{N-1}) + C_{0} \frac{1-x}{-\alpha_{N-1}+1} - C_{0}
\frac{(1-x_{N-1})^{-\alpha_{N-1}+1}(1-x)^{\alpha_{N-1}}}{-\alpha_{N-1}+1} \\
& -(C_{0}+C_{1})\frac{1}{-\alpha_{N-1}} + (C_{0}+C_{1})\frac{(1-x_{N-1})^{-\alpha_{N-1}}(1-x)^{\alpha_{N-1}}}{-\alpha_{N-1}}.
\end{eqnarray*}

Letting $ x \to x_{N} = 1$ and making use of \eqref{FluxODERightBC} we arrive at $v(x_{N}) = u_{N} = \frac{ C_{0} + C_{1} }{\alpha_{N-1}}$ and finally

\begin{eqnarray}\label{FluxRight1}
\begin{split}
& v(x) = \frac{ (1-x)^{\alpha_{N-1} }}{(1-x_{N-1})^{\alpha_{N-1}}}(u_{N-1} - u_{N}) + u_{N} \\
& + \frac{\omega}{-\alpha_{N-1} + 1}(1-x)\left(1 - \frac{(1-x)^{\alpha_{N-1} - 1}}{( 1-x_{N-1} )^{\alpha_{N-1} - 1}}\right),
\end{split}
\end{eqnarray}
where $ \omega = C_{0} \in R $ is a free parameter. Therefore

\begin{equation*}
\rho_{N-1} (v) := { \stackrel{=}{a} }_{N-1/2}(x-1)\omega + b_{N-1/2} u_{N}.
\end{equation*}

\textbf{Case 3.2} $ \alpha_{N-1} = 1 $.

 Now we solve the following ODE

\begin{equation*}
\left( \frac{v} {1-x} \right)^{'} = C_{0} (1-x)^{-2}x + C_{1}(1-x)^{-2}.
\end{equation*}

Integrating over $ (x_{N-1},x), x \in I_{N-1} $, we obtain

\begin{eqnarray*}
& v(x) = \frac{1-x}{1-x_{N-1}}v(x_{N-1})+C_{0}(1-x)(\ln (1-x) - \ln (1-x_{N-1})) \\
& + (C_{0} + C_{1}) \left( \frac{1}{1-x} - \frac{1}{1 - x_{N-1}} \right)(1-x).
\end{eqnarray*}

Letting $x \to x_{N} = 1$ one gets

\begin{equation*}
v (x_{N} ) = u_{N} = C_{0} + C_{1} = \frac{ C_{0} + C_{1} }{\alpha_{N-1}},
\end{equation*}

\begin{equation*}
v(x) = \frac{1-x}{1 - x_{N-1}}(u_{N-1} - u_{N}) + u_{N} + \omega (1-x) \ln \frac{1-x}{1 - x_{N-1}}.
\end{equation*}

Since $ 1-x > 0 $ we can conclude that this is the result of the limiting process $ \alpha_{N-1} \to 1 $, performed on \eqref{FluxRight1}. The flux in both cases 3.1 and 3.2 can be written in the form

\begin{equation*}
\rho_{N-1} (v) = {\stackrel{=}{a}}_{N-1/2} (x-1) \omega + b_{N-1/2}u_{N}.
\end{equation*}

\textbf{Case 3.3} $ \alpha_{N-1} = 0 $.

Integrating over $ (x_{N-1},x), x \in I_{N-1} $, the following ODE

\begin{equation*}
v' = C_{0} (1-x)^{-1}x + C_{1}(1-x)^{-1}
\end{equation*}
we arrive at

\begin{equation*}
v(x) = v ( x_{N-1} ) - C_{0} (x - x_{N-1}) - (C_{0} + C_{1})(\ln (1-x) - \ln (1 - x_{N-1})).
\end{equation*}

The function $v(x)$ is bounded for $ x \to x_{N} $ when $C_{0} + C_{1} = 0$. Therefore $C_{0} = \frac{u_{N-1} - u_{N}}{1-x_{N-1}}$ and

\begin{equation*}
v(x) = u_{N} + (u_{N-1} - u_{N})\frac{1-x} {1-x_{N-1}}.
\end{equation*}

The flux has the following form

\begin{equation*}
\rho_{N-1} (v) = {\stackrel{=}{a}}_{N-1/2} (1-\alpha_{N-1})(1-x)\frac{ u_{N} - u_{N-1}}{h_{N-1}} + b_{N-1/2} u_{N},
\end{equation*}
where we used that $\alpha_{N-1} = b_{N-1/2} = 0$.

\textbf{Case 3.4} $ \alpha_{N-1} < 0 $.

This time we integrate \eqref{FluxODERight3} from $x$ to $ x_{N}=1$ and obtain

\begin{equation*}
v(x) = C_{0} \frac{1-x}{1-\alpha_{N-1}} + (C_{0} + C_{1})\frac{1}{\alpha_{N-1}}
\end{equation*}
and using the boundary conditions

\begin{equation*}
C_{0} = \frac{u_{N-1} - u_{N}}{1-x_{N-1}}(1 - \alpha_{N-1}), \; C_{1} = \alpha_{N-1} u_{N} - \frac{u_{N-1} - u_{N}}{1-x_{N-1}}(1-\alpha_{N-1}).
\end{equation*}

Therefore

\begin{equation*}
v(x) = \frac{1-x}{1 - x_{N+1}}( u_{N-1} - u_{N} ) + u_{N},
\end{equation*}

\begin{equation*}
\rho_{N-1} (v) = {\stackrel{=}{a}}_{N-1/2} (1-\alpha_{N-1})(1-x) \frac{ u_{N} - u_{N-1}}{h_{N-1}} + b_{N-1/2} u_{N}
\end{equation*}
and these are exactly the same results as in Case 3.3. Finally, a reasonable choice of the free parameter $\omega$ is 0 and

\begin{equation*}
v(x) = \left\{
    \begin{array}{ll}
        u_{N} + \frac{1-x}{1-x_{N-1}}(u_{N-1} - u_{N}), &  \alpha_{N-1} \leq 0, \\
        u_{N} + \frac{(1-x)^{\alpha_{N-1}}}{({1 - x_{N-1}})^{\alpha_{N-1}}}( u_{N-1} - u_{N} ), & \alpha_{N-1}>0,
    \end{array}
    \right.
\end{equation*}

\begin{equation}\label{FluxRight}
\rho_{N-1} (v) = \left\{
    \begin{array}{ll}
        {\stackrel{=}{a}}_{N-1/2}(1-\alpha_{N-1})(1-x)\frac{u_{N} - u_{N-1}}{h_{N-1}} + b_{N-1/2}u_{N}, & \alpha_{N-1} \leq 0, \\
        b_{N-1/2}u_{N}, & \alpha_{N-1} > 0.
    \end{array}
    \right.
\end{equation}

Let us introduce the finite element space $V_{h}$ by specifying it's basis $\{\phi_{i}\}_{i=0}^{N}$. Following \cite{CV,RV} we introduce the functions

\begin{eqnarray*}
\phi_{i}(x) = \left\{
     \begin{array}{ll}
      \frac{\left( \frac{1} {x_{i-1}} - 1 \right)^{\alpha_{i-1}} - \left( \frac{1}{x} - 1 \right)^{\alpha_{i-1}}}
      {\left( \frac{1} {x_{i-1}} - 1 \right)^{\alpha_{i-1}} - \left( \frac{1}{x_{i} } - 1 \right)^{ \alpha_{i-1}}}, & x \in (x_{i-1},x_{i}), \\
      \frac{\left( \frac{1} {x_{i+1}} - 1 \right)^{\alpha_{i}} - \left( \frac{1}{x} - 1 \right)^{\alpha_{i}}}
      {\left( \frac{1} {x_{i+1}} - 1 \right)^{\alpha_{i}} - \left(\frac{1}{x_{i}} - 1 \right)^{\alpha_{i}}}, & x \in (x_{i},x_{i+1}).
     \end{array}
  \right.
\end{eqnarray*}

On the intervals $(0,x_{1})$ and $(x_{N-1},1)$ we define the linear functions

\begin{eqnarray*}
\phi_{0} (x) = \left\{
    \begin{array}{l}
     1 - \frac{x}{ x_{1}}, \; x \in (0, x_{1}) \\ 0, \; \mbox{otherwise};
    \end{array}
    \right. , \; \;
\phi_{N} (x) = \left\{
    \begin{array}{l}
     \frac{x-x_{N-1}}{1-x_{N-1}}, \; x \in (x_{N-1},1), \\ 0, \; \mbox{otherwise}.
    \end{array}
   \right.
\end{eqnarray*}

Next we define the linear functions $\phi_{1}(x)$ and $\phi_{N-1}(x)$ on the intervals $(0,x_{2})$ and $(x_{N-2},1)$
\begin{eqnarray*}
\phi_{1} (x) = \left\{
    \begin{array}{l}
     1 - \frac{x}{ x_{1}} , \; x \in (0, x_{1}), \\
     \frac{\left( \frac{1} {x_{2}} - 1 \right)^{\alpha_{1}} - \left( \frac{1}{x} - 1 \right)^{\alpha_{1}}}
     {\left(\frac{1} {x_{2}} - 1 \right)^{\alpha_{1}} - \left(\frac{1}{x_{1}} - 1 \right)^{\alpha_{1}}},
     \; \; x \in (x_{1},x_{2}); \\ 0, \; \mbox{otherwise};
    \end{array}
    \right.
\end{eqnarray*}
\begin{eqnarray*}
\phi_{N-1} (x) = \left\{
    \begin{array}{l}
     \frac{\left(\frac{1}{x_{N-2}} - 1 \right)^{\alpha_{N-2}} - \left(\frac{1}{x} - 1 \right)^{\alpha_{N-2}}}
     {\left(\frac{1} {x_{N-2}} - 1 \right)^{\alpha_{N-2}} - \left( \frac{1}{x_{N-1}} - 1 \right)^{\alpha_{N-2}}}, \; \; x \in (x_{N-2} ,x_{N-1}), \\
     (1-x)/(1-x_{N-1}) , \; x \in (x_{N-1} ,1) ; \\ 0, \; \mbox{otherwise}.
    \end{array}
    \right.
\end{eqnarray*}

Then, for any $v_{h} \in V_{h}$, we have the representation $v_{h} (x) = \sum_{i=0}^{N} v_{hi} \phi_{i}$, where $v_{hi}:= v_{h} ( x_{i})$. Associated with $v_{h}$, we introduce the natural interpolation operator $I_{h}: C({\overline{\Omega}}) \to V_{h}$ by

\begin{equation*}
I_{h} v_{h} (x_{i}) := v_{h} (x_{i}) = v_{hi}, \; \; i=0,\dots,N.
\end{equation*}

Furthermore, using the flux approximations \eqref{FluxMid},\eqref{FluxLeft},\eqref{FluxRight}, obtained in Cases 1, 2 and 3 respectively, we define by $\rho_{h}(u)$ an approximation to
$\rho (u) \vert_{I_{i}} := \rho_{i} (u), \; \; i=0,\dots,N-1$. Coming back to \eqref{VarForm}, this motivates the following semi-discretization of \eqref{DivEq} in the space $V_{h}$:
\begin{eqnarray*}
& ( {\dot u}_{h}(t), v_{h} )_{h} + A_{h}( u_{h}(t), v_{h}; t) = 0 \; \; \forall v_{h} \in V_{h}, \\
& A_{h} (w_{h},v_{h};t) := - \sum_{i=0}^{N} \left[ x(1-x) \rho_{h}(w_{h},x,t) \right]_{x_{i-1/2}}^{x_{i+1/2}} L_{h} v_{h}(x_{i}) + ( c(t)w_{h},v_{h} )_{h}.
\end{eqnarray*}

As usual, from \eqref{VarForm} an equivalent ODEs system is obtained by setting successfully $v_{h} = \phi_{i}, \; i=0,\dots,N$:
\begin{eqnarray*}
& {\dot u}_{hi}(t) l_{i} - x_{i+1/2}(1-x_{i+1/2})\rho_{h}(u_{h}(t),x_{i+1/2},t) \\
& + x_{i-1/2}(1-x_{i-1/2})\rho_{h}(u_{h}(t),x_{i-1/2},t) + c_{i}(t) u_{hi}(t) l_{i} = 0, \; \; c_{i}(t) := c(x_{i},t).
\end{eqnarray*}

The complete set of equations forms an $(N+1) \times (N+1)$ system of linear ODEs w.r.t. $u_{h} (t) := (u_{h0}(t),\dots,u_{hN}(t))^{T}$:
\begin{equation*}
M_{h}{\dot u}_{h}(t) + A_{h} (t) u_{h} (t)=0,
\end{equation*}
where
\begin{equation*}
M_{h} := \left((\phi_{j},\phi_{i})_{h}\right)_{i,j=0}^{N} = diag (l_{0},\dots,l_{N}),
\end{equation*}
\begin{equation*}
A_{h} (t) := \left(A_{h}(\phi_{j},\phi_{i};t)\right)_{i,j=0}^{N} = \left(a_{ij}(t)\right)_{i,j=0}^{N}.
\end{equation*}

\section{Full discretization}

Let $0 =: t_{0} < t_{1} < \dots < t_{N_t} := T$ be a subdivision of the time interval $[0,T]$ with the step sizes $\triangle t_{m} := t_{m+1} - t_{m} > 0, \; m \in \{0,\dots,N_t-1\}$. The fully discrete method with parameter $\theta \in [0,1]$ for \eqref{DivEq} reads as follows:

\emph{Find a sequence $U^{1},\dots,U^{N_t} \in V_{h}$ such that for $m \in \{0,\dots,N_t-1\}$}
\begin{eqnarray*}
& \left(\frac{U^{m+1} - U^{m}}{\triangle t_{m}},v_{h}\right) + A_{h}(\theta U^{m+1} + (1-\theta)U^{m},v_{h};t_{m+\theta})=0\;\forall v_{h} \in V_{h}, \\
& U^{0} = u_{0h},
\end{eqnarray*}
\emph{where $t_{m+\theta} := \theta t_{m+1} + (1 - \theta)t_{m} = t_{m} + \theta \triangle t_{m}$ and $ u_{0h} \in V_{h} $ is an approximation to $u_{0}$}. By representing the elements $U^{m} $ in terms of the basis $ \{ \phi_{i} \}_{i=0}^{N-1} $ of $ V_{h} $ and choosing $v_{h} = \phi_{j}, \; j=0,\dots,N$ we get the algebraic form

\begin{equation}\label{FullDForm}
\frac{\textbf{M}_{h} \textbf{u}^{m+1}_{h} - \textbf{M}_{h} \textbf{u}^{m}_{h}}{ \triangle t_{m} } + \theta \textbf{A}^{m+\theta}_{h} \textbf{u}_{h}^{m+1} + (1 - \theta) \textbf{A}^{m+\theta}_{h} \textbf{u}_{h}^{m}=0, \;
\textbf{A}_{h}^{m}:= \textbf{A}_{h}(t_{m}).
\end{equation}

The initial condition $\textbf{u}_{h}^{0}$ is obtained from the representation of $u_{0h}$ by means of the basis of $V_{h}$.

We will show, Theorem \ref{Mmatrix}, that the system matrix $\textbf{E}_h = \{e_{i,j}\}_{i,j=0}^{N,N_t} = \textbf{M}_h/\triangle t_{m}+\theta \textbf{A}^{m+\theta}_{h}$ can be reduced to an $M$-matrix by excluding the first two and the last two equations in \eqref{FullDForm}. Therefore, the above problem \eqref{FullDForm} is uniquely solvable and our method \emph{preserves the positivity}, Theorem \ref{WMP} (maximum principle), of the numerical solution in time. Let us introduce the notations
\begin{eqnarray*}
\varphi_i^{\alpha_i} && :=\left( \frac{x_i}{1-x_i} \right)^{\alpha_i}, \; \Delta_{i}^{\alpha_i} := \frac{1}{\varphi_{i+1}^{\alpha_i}-\varphi_{i}^{\alpha_i}}, \\
a_{i \pm 1/2} && := a(x_{i \pm 1/2}), \; b_{i \pm 1/2}^{m+\theta} := b(x_{i \pm 1/2},t^{m+\theta}), \; c_{i \pm 1/2}^{m+\theta} := c(x_{i \pm 1/2},t^{m+\theta})
\end{eqnarray*}
and write down the elements of the system matrix:

for $i=0$ if $\alpha_0 < 0$ then
\begin{eqnarray*}
e_{0,0} && = \frac{l_{0}}{\Delta t_m} + \theta x_{1/2}(1-x_{1/2})(-b_{1/2}^{m+\theta}) + \theta l_{0}c_{0}^{m+\theta}, \; e_{0,1} = 0 \\
F_{0} &&  = u_{h0}^{m}\left(\frac{l_{0}}{\Delta t_m} - (1-\theta)\left(x_{1/2}(1-x_{1/2})(-b_{1/2}^{m+\theta}) + l_{0}c_{0}^{m+\theta}\right)\right),
\end{eqnarray*}
and if $\alpha_0 \geq 0$ then
\begin{eqnarray*}
e_{0,0} && = \frac{l_{0}}{\Delta t_m} + \theta x_{1/2}(1-x_{1/2})0.5(\overline{a}_{1/2}-b_{1/2}^{m+\theta}) + \theta l_{0}c_0^{m+\theta}, \\
e_{0,1} && = - \theta x_{1/2}(1-x_{1/2})0.5(\overline{a}_{1/2}+b_{1/2}^{m+\theta}), \\
F_{0} && = u_{h0}^{m}\left( \frac{l_{0}}{\Delta t_m} - (1-\theta)\left( x_{1/2}(1-x_{1/2})0.5(\overline{a}_{1/2}-b_{1/2}^{m+\theta}) \right. \right. \\
&& + \left. \left. l_{0}c_0^{m+\theta} \right) \right) + u_{h1}^{m}(1-\theta)\left( x_{1/2}(1-x_{1/2})0.5(\overline{a}_{1/2}+b_{1/2}^{m+\theta}) \right);
\end{eqnarray*}
for $i=1$ if $\alpha_0 < 0$ then
\begin{eqnarray*}
e_{1,1} && = \frac{l_{1}}{\Delta t_m} + \theta x_{i+1/2}(1-x_{i+1/2})b_{i+1/2}^{m+\theta}\phi_{i}^{\alpha_i}\Delta_{i}^{\alpha_i} + \theta l_{1}c_1^{m+\theta}, \\
e_{1,0} && = - \theta x_{1/2}(1-x_{1/2})(-b_{1/2}^{m+\theta}), \\
e_{1,2} && = - \theta x_{i+1/2}(1-x_{i+1/2}) b_{i+1/2}^{m+\theta}\phi_{i+1}^{\alpha_i}\Delta_{i}^{\alpha_i}, \\
F_{1} && = u_{h1}^{m}\left( \frac{l_{1}}{\Delta t_m} - (1-\theta)\left( x_{i+1/2}(1-x_{i+1/2})b_{i+1/2}^{m+\theta}\phi_{i}^{\alpha_i}\Delta_{i}^{\alpha_i} \right. \right. \\ && \left. \left. +  l_{1}c_1^{m+\theta} \right) \right) + u_{h0}^{m} (1-\theta) x_{1/2}(1-x_{1/2})(-b_{i+1/2}^{m+\theta}) \\
&& + u_{h2}^{m} (1-\theta) x_{i+1/2}(1-x_{i+1/2}) b_{i+1/2}^{m+\theta}\phi_{i+1}^{\alpha_i}\Delta_{i}^{\alpha_i}
\end{eqnarray*}
and if $\alpha_0 \geq 0$ then
\begin{eqnarray*}
e_{1,1} && = \frac{l_{1}}{\Delta t_m} + \theta x_{i+1/2}(1-x_{i+1/2})b_{i+1/2}^{m+\theta}\phi_{i}^{\alpha_i}\Delta_{i}^{\alpha_i} \\
&& + \theta x_{1/2}(1-x_{1/2})0.5(\overline{a}_{1/2}+b_{i+1/2}^{m+\theta}) + \theta l_{1}c_1^{m+\theta}, \\
e_{1,0} && = - \theta x_{1/2}(1-x_{1/2})0.5(\overline{a}_{1/2}-b_{i+1/2}^{m+\theta}), \\
e_{1,2} && = - \theta x_{i+1/2}(1-x_{i+1/2}) b_{i+1/2}^{m+\theta} \phi_{i+1}^{\alpha_i}\Delta_{i}^{\alpha_i}, \\
F_{1} && = u_{h1}^{m}\left( \frac{l_{1}}{\Delta t_m} - (1-\theta)\left( x_{i+1/2}(1-x_{i+1/2})b_{i+1/2}^{m+\theta}\phi_{i}^{\alpha_i}\Delta_{i}^{\alpha_i} \right. \right. \\
&& \left. \left. + x_{1/2}(1-x_{1/2})0.5(\overline{a}_{1/2}+b_{i+1/2}^{m+\theta}) + l_{1}c_1^{m+\theta} \right) \right) \\
&& + u_{h0}^{m} (1-\theta) x_{1/2}(1-x_{1/2})0.5(\overline{a}_{1/2}-b_{i+1/2}^{m+\theta}) \\
&& + u_{h2}^{m} (1-\theta) x_{i+1/2}(1-x_{i+1/2}) b_{i+1/2}^{m+\theta} \phi_{i+1}^{\alpha_i}\Delta_{i}^{\alpha_i};
\end{eqnarray*}
for $i=2,\dots,N-2$
\begin{eqnarray*}
e_{i,i} && = \frac{l_{i}}{\Delta t_m} + \theta x_{i+1/2}(1-x_{i+1/2})b_{i+1/2}^{m+\theta}\phi_{i}^{\alpha_i}\Delta_{i}^{\alpha_i} \\
&& + \theta x_{i-1/2}(1-x_{i-1/2})b_{i-1/2}^{m+\theta}\phi_{i}^{\alpha_{i-1}}\Delta_{i-1}^{\alpha_{i-1}} + \theta l_{i} c_i^{m+\theta}, \\
e_{i,i-1} && = - \theta x_{i-1/2}(1-x_{i-1/2})b_{i-1/2}^{m+\theta}\phi_{i-1}^{\alpha_{i-1}}\Delta_{i-1}^{\alpha_{i-1}},
\end{eqnarray*}
\begin{eqnarray*}
e_{i,i+1} && = - \theta x_{i+1/2}(1-x_{i+1/2})b_{i+1/2}^{m+\theta}\phi_{i+1}^{\alpha_i}\Delta_{i}^{\alpha_i},  \\
F_{i} && = u_{hi}^{m}\left( \frac{l_{i}}{\Delta t_m} - (1-\theta)\left( x_{i+1/2}(1-x_{i+1/2})b_{i+1/2}^{m+\theta}\phi_{i}^{\alpha_i}\Delta_{i}^{\alpha_i} \right. \right. \\
&& \left. \left. + x_{i-1/2}(1-x_{i-1/2})b_{i-1/2}^{m+\theta}\phi_{i}^{\alpha_{i-1}}\Delta_{i-1}^{\alpha_{i-1}} + l_i c_i^{m+\theta} \right) \right) \\
&& + u_{hi-1}^{m} (1-\theta)x_{i-1/2}(1-x_{i-1/2})b_{i-1/2}^{m+\theta}\phi_{i-1}^{\alpha_{i-1}}\Delta_{i-1}^{\alpha_{i-1}} \\
&& + u_{hi+1}^{m} (1-\theta)x_{i+1/2}(1-x_{i+1/2})b_{i+1/2}^{m+\theta}\phi_{i+1}^{\alpha_i}\Delta_{i}^{\alpha_i};
\end{eqnarray*}
for $i=N-1$ if $\alpha_{N-1}>0$
\begin{eqnarray*}
e_{N-1,N-1} && = \frac{l_{N-1}}{\Delta t_m} + \theta x_{i-1/2}(1-x_{i-1/2})b_{i-1/2}^{m+\theta}\phi_{i}^{\alpha_{i-1}}\Delta_{i-1}^{\alpha_{i-1}} + \theta l_{N-1}c_{N-1}^{m+\theta}, \\
e_{N-1,N-2} && = - \theta x_{N-1/2}(1-x_{N-1/2})b_{N-1/2}^{m+\theta}, \\
e_{N-1,N} && = - \theta x_{i-1/2}(1-x_{i-1/2})b_{i-1/2}^{m+\theta} \phi_{i-1}^{\alpha_{i-1}}\Delta_{i-1}^{\alpha_{i-1}}, \\
F_{N-1} && = u_{hN-1}^{m}\left( \frac{l_{1}}{\Delta t_m} - (1-\theta) \left( x_{i-1/2}(1-x_{i-1/2})b_{i-1/2}^{m+\theta}\phi_{i}^{\alpha_{i-1}}\Delta_{i-1}^{\alpha_{i-1}} \right. \right. \\
&& \left. \left. + x_{N-1/2}(1-x_{N-1/2})b_{N-1/2}^{m+\theta} + l_{N-1}c_{N-1}^{m+\theta} \right) \right) \\
&& + u_{hN-2}^{m}(1-\theta)x_{N-1/2}(1-x_{N-1/2})b_{N-1/2}^{m+\theta} \\
&& + u_{hN}^{m}(1-\theta)x_{i-1/2}(1-x_{i-1/2})b_{N-1/2}^{m+\theta}\phi_{i-1}^{\alpha_{i-1}}\Delta_{i-1}^{\alpha_{i-1}}
\end{eqnarray*}
and if $\alpha_{N-1} \leq 0$ then
\begin{eqnarray*}
e_{N-1,N-1} && = \frac{l_{N-1}}{\Delta t_m} + \theta x_{i-1/2}(1-x_{i-1/2})b_{i-1/2}^{m+\theta}\phi_{i}^{\alpha_{i-1}}\Delta_{i-1}^{\alpha_{i-1}} \\
&& + \theta x_{N-1/2}(1-x_{N-1/2})0.5({\stackrel{=}{a}}_{i-1/2}^{m+\theta}-b_{i-1/2}^{m+\theta}) + \theta l_{N-1}c_{N-1}^{m+\theta}, \\
e_{N-1,N} && = - \theta x_{N-1/2}(1-x_{N-1/2})0.5({\stackrel{=}{a}}_{i-1/2}^{m+\theta}+b_{i-1/2}^{m+\theta}), \\
e_{N-1,N-2} && = - \theta x_{i-1/2}(1-x_{i-1/2})b_{i-1/2}^{m+\theta}\phi_{i-1}^{\alpha_{i-1}}\Delta_{i-1}^{\alpha_{i-1}},  \\
F_{N-1} && = u_{hN-1}^{m}\left( \frac{l_{N-1}}{\Delta t_m} - (1-\theta) \left( x_{i-1/2}(1-x_{i-1/2})b_{i-1/2}^{m+\theta}\phi_{i}^{\alpha_{i-1}}\Delta_{i-1}^{\alpha_{i-1}} \right. \right. \\
&& \left. \left. + x_{N-1/2}(1-x_{N-1/2})0.5({\stackrel{=}{a}}_{i-1/2}^{m+\theta}-b_{i-1/2}^{m+\theta}) + l_{N-1}c_{N-1}^{m+\theta} \right) \right) \\
&& + u_{hN-2}^{m}(1-\theta)x_{i-1/2}(1-x_{i-1/2})b_{i-1/2}^{m+\theta}\phi_{i-1}^{\alpha_{i-1}}\Delta_{i-1}^{\alpha_{i-1}} \\
&& + u_{hN}^{m}(1-\theta)x_{N-1/2}(1-x_{N-1/2})0.5({\stackrel{=}{a}}_{i-1/2}^{m+\theta}+b_{N-1/2}^{m+\theta});
\end{eqnarray*}
for $i=N$ if $\alpha_{N-1} > 0$ then
\begin{eqnarray*}
e_{N,N} && = \frac{l_{N}}{\Delta t_m} + \theta x_{N-1/2}(1-x_{N-1/2})b_{N-1/2}^{m+\theta} + \theta l_{N}c_N^{m+\theta}, \;
e_{N,N-1} = 0, \\
F_{N} && = u_{hN}^{m}\left( \frac{l_{N}}{\Delta t_m} - (1-\theta) \left( x_{N-1/2}(1-x_{N-1/2})b_{N-1/2}^{m+\theta} + l_{N}c_N^{m+\theta} \right) \right)
\end{eqnarray*}
and if $\alpha_{N-1} \leq 0$ then
\begin{eqnarray*}
e_{N,N} && = \frac{l_{N}}{\Delta t_m} + \theta x_{N-1/2}(1-x_{N-1/2})0.5({\stackrel{=}{a}}_{N-1/2}^{m+\theta}+b_{N-1/2}^{m+\theta}) + \theta l_{N}c_N^{m+\theta}, \\
e_{N,N-1} && = - \theta x_{N-1/2}(1-x_{N-1/2})0.5({\stackrel{=}{a}}_{N-1/2}^{m+\theta}-b_{N-1/2}^{m+\theta}), \\
F_{N} && = u_{hN}^{m}\left( \frac{l_{N}}{\Delta t_m} - (1-\theta) \left( x_{N-1/2}(1-x_{N-1/2})0.5({\stackrel{=}{a}}_{N-1/2}^{m+\theta} \right. \right. \\
&& \left. \left. + b_{N-1/2}^{m+\theta}) + l_{N}c_N^{m+\theta} \right) \right) \\
&& + u_{hN-1}^{m}(1-\theta) x_{N-1/2}(1-x_{N-1/2})0.5({\stackrel{=}{a}}_{N-1/2}^{m+\theta}-b_{N-1/2}^{m+\theta}).
\end{eqnarray*}

\begin{theorem}\label{Mmatrix}
For any given $m=1,2,\dots,N_t$, if $\Delta t_m$ is sufficiently small, the system matrix of \eqref{FullDForm}, $\textbf{E}_h$, is an $M$-matrix.
\end{theorem}

\begin{proof}
Let us write down the scalar form of \eqref{FullDForm}:

\begin{eqnarray*}
& B_{0} u_{h0}^{m+1} + C_{0}u_{h1}^{m+1} = F_{0} \\
& A_{1} u_{h0}^{m+1} + B_{1} u_{h1}^{m+1} + C_{1} u_{h2}^{m+1} = F_{1} \\
& A_{2} u_{h1}^{m+1} + B_{2} u_{h2}^{m+1} + C_{2} u_{h3}^{m+1} = F_{2} \\
& ............................... \\
& A_{i} u_{hi-1}^{m+1} + B_{i} u_{hi}^{m+1} + C_{i} u_{hi+1}^{m+1} = F_{i}, \\
& ............................... \\
& A_{N} u_{hN-1}^{m+1} + B_{N} u_{hN}^{m+1} = F_{N},
\end{eqnarray*}
for $i=3,4,\dots,N-1$ and
\begin{eqnarray*}
& B_{0} = \frac{h_{0}}{2 \Delta t_m} + \theta e_{0,0}, \; C_{0} = - \theta e_{0,1}, \\
& A_{1} = - \theta e_{1,0}, \; B_{1} = \frac{l_{1}}{\Delta t_m} + \theta e_{1,1}, \; C_{1} = - \theta e_{1,2}, \\
& A_{i} = - \theta e_{i,i-1}, \; B_{i} = \frac{l_{i}}{\Delta t_m} + \theta e_{i,i}, \; C_{i} = - \theta e_{i,i+1}, \; i=2,3,\ldots,N-1, \\
& A_{N} = - \theta e_{N,N-1}, \; B_{N} = \frac{h_{N-1}}{2 \Delta t_m} + \theta e_{N,N}, \\
& F_{0} = \left( \frac{h_{0}}{2 \Delta t_m} - (1 - \theta) e_{0,0} \right) u_{h0}^{m} + (1 - \theta) e_{0,1} u_{h1}^{m}, \\
& F_{1} = (1 - \theta) e_{1,0} u_{h0}^{m} + \left( \frac{l_{1}}{\Delta t_m} - (1 - \theta) e_{1,1} \right) u_{h1}^{m} + (1 - \theta) e_{1,2} u_{h2}^{m}, \\
& F_{i} = (1 - \theta) e_{i,i-1} u_{hi-1}^{m} + \left( \frac{l_{i}}{\Delta t_m} - (1 - \theta) e_{i,i} \right) u_{hi}^{m} + (1 - \theta) e_{i,i+1} u_{hi+1}^{m}, \\
& F_{N} = (1 - \theta) e_{N,N-1} u_{hN-1}^{m} + \left( \frac{h_{N-1}}{2 \Delta t_m} - (1 - \theta) e_{N,N} \right) u_{hN}^{m}.
\end{eqnarray*}

Let us first investigate the off-diagonal entries of the system matrix $A_{i} = - \theta e_{i,i-1}$ and $C_{i} = - \theta e_{i,i+1}$. From the formulas for $ e_{i,j} $ from the above we have  $e_{i,j}> 0$, $i,j = 1,2,\dots,N-1, i \ne j.$ That is because
\begin{eqnarray*}
& b_{i+1/2} \frac { \left( \frac{x_{i+1}}{1 - x_{i+1}} \right)^{\alpha_{i}}}{ \left( \frac{x_{i+1}}{1 - x_{i+1}} \right)^{\alpha_{i}} -                                                             \left( \frac{x_{i}}{1 - x_{i}} \right)^{\alpha_{i}}} = a_{i+1/2}\alpha_{i}\frac { \left( \frac{x_{i+1}}{1 - x_{i+1}} \right)^{\alpha_{i}}}
{ \left( \frac{x_{i+1}}{1 - x_{i+1}} \right)^{\alpha_{i}} - \left( \frac{x_{i}}{1 - x_{i}} \right)^{\alpha_{i}}} \\
& = a_{i+1/2}\frac{\alpha_{i}}{1-{\overline{x}}_{i}^{\alpha_{i}}}>0, \; 0<{\overline{x}}_{i} = \frac{x_{i}}{x_{i+1}}\frac{1-x_{i+1}}{1 - x_{i}} < 1
\end{eqnarray*}
for each $ i=1,2,\dots,N-1 $ and each $b_{i+1/2} \not = 0$. We have used that $1-{\overline{x}}_{i}^{\alpha_{i}}$ has just the sign of $\alpha_{i} $
and this is also true for $ b_{i+1/2} \to 0$. Now it is clear that $A_{i} = - \theta e_{i,i-1}$ and $C_{i} = - \theta e_{i,i+1}$ are negative.

We should also note that $B_{i}$ is always positive since $\Delta t_m$ is small. The situation is different for $B_{0}$, $C_{0}$, $A_{1}$, $B_{1}$, $C_{1}$
and $A_{N-1}$, $B_{N-1}$, $C_{N-1}$, $A_{N}$, $B_{N}$. From the first three equations we find

\begin{eqnarray*}
& u_{h0}^{m+1} = \frac{F_{0}}{B_{0}} - \frac{C_{0}}{B_{0}} u_{h1}^{m+1}, \;
u_{h1}^{m+1} = \frac{\triangle_{1}}{\triangle} - \frac{C_{1}}{\triangle} u_{h2}^{m+1}, \\
& \triangle = B_{1} - \frac{A_{1}}{B_{0}} C_{0}, \; \triangle_{1} = F_{1} - \frac{A_{1}}{B_{0}} F_{0}, \\
& {\widetilde {B}}_{2} u_{h2}^{m+1} + C_{2} u_{h3}^{m+1} = {\widetilde {F}}_{2}, \\
& {\widetilde {B}}_{2} = B_{2} - \frac{A_{2} C_{1}}{\triangle}, \; {\widetilde {F}}_{2} = F_{2} - \frac{\triangle_{1}}{\triangle} A_{2}.
\end{eqnarray*}

It is easily to see that when $ \triangle > 0 $ and $ \triangle = O \left( \frac{1}{\Delta t_m} \right) $ then $ B_{2} = O \left( \frac{1}{\Delta t_m} \right) $
for small $ \Delta t_m $. Therefore $ {\widetilde {B}}_{2} = O \left(\frac{1}{\Delta t_m} \right)$ and ${\widetilde {B}}_{2} > \vert C_{2} \vert $.

In a similar way one can eliminate $ u_{hN-1}^{m+1}$ and $u_{hN}^{m+1} $. As a result we obtain a system of linear algebraic equations with unknowns
$ u_{h2}^{m+1},\dots,u_{hN-2}^{m+1} $ and coefficients matrix that is an $M$-matrix.

While $F_{3},...,F_{N-3}$ are non-negative, we have to prove that $\widetilde {F}_{2}$ and $\widetilde {F}_{N-2}$ are also non-negative. From the formula for $\widetilde {F}_{2}$ it follows that when $\Delta t_m$ is small $\widetilde {F}_{2}$ is non-negative since $F_{2}=O \left( \frac{1} {\Delta t_m}  \right) $ and ${ \triangle },{ \triangle_{1} }$ are of the same order with respect to $\Delta t_m$. $\widetilde {F}_{N-2}$ is handled by the same way as $\widetilde {F}_{2}$ and also considered non-negative.

Since the load vector $( \widetilde {F}_{2},F_{3}, \dots ,F_{N-3},\widetilde {F}_{N-2} )$ is non-negative and the corresponding matrix is an M-matrix we can conclude that $ u_{h2}^{m+1},\dots,u_{hN-2}^{m+1} $ are non-negative. Finally, using the formulas for $u_{h0}^{m+1},u_{h1}^{m+1},u_{hN-1}^{m+1},u_{hN}^{m+1}$ one can easily check that they are non-negative too if $\Delta t_m$ is small.
\end{proof}

\begin{remark}
Theorem \ref{Mmatrix} shows that the fully-discretized system \eqref{FullDForm} satisfies the discrete maximum principle and therefore the above discretization is monotone. This guarantees the following: for a non-negative initial function $u_0(x)$ the numerical solution $\textbf{u}_{h}^{m}$, obtained via this method, is also non-negative as expected, because the price of the option is a non-negative number.
\end{remark}

\section{Numerical experiments}

Numerical experiments, presented in this section, illustrate the properties of the constructed method. We solve numerically various European Test Problems (TP) with different final (initial) conditions and different choices of parameters.

\begin{enumerate}
  \item ($TP1$). \emph{Call option} with final condition \eqref{FinalC1}. Parameters: $S_{max}=700$, $T=1$, $r=0.1$, $\sigma=0.3$, $d=0.04$ and $E=400$.
  \item ($TP2$). \emph{Call option} with cash-or-nothing payoff $V(S,T)=H(S-E), \; S \in (0,\infty)$, where $H$ denotes the Heaviside function.
Parameters: $S_{max}=700$, $T=1$, $r=0.1$, $\sigma=0.4$, $d=0.04$ and $E=400$.
  \item ($TP3$). \emph{Call option} with final condition \eqref{FinalC1}. Parameters: $S_{max}=700$, $T=1$, $r=0.1+0.02sin(10Tt)$, $\sigma=0.4$, $d=0.06S/(S+E)$ and $E=400$.
  \item ($TP4$). \emph{A portfolio of options}. Combinations of different options have step final conditions such as the 'bullish vertical spread' payoff, defined in \eqref{FinalC2}. In this example, we assume that the final condition is a 'butterfly spread' delta function, defined by
\begin{equation*}
V(S,T)=\left\{
       \begin{array}{ll}
         1 , & S \in (S_1,S_2), \\
         -1 , & S \in (S_2,S_3), \\
         0 , & otherwise,
       \end{array}
     \right.
\end{equation*}
which arises from a portfolio of three types of options with different exercise prices \cite{WHD}. Parameters: $S_{max}=100$, $T=1$, $S_1=40$, $S_2=50$, $S_3=60$, $r=0.1+0.02 sin(10Tt)$, $\sigma=0.4$, $d=0.06S/(S+E)$ and $P_m=50$.
\end{enumerate}

In the tables below are presented the computed $C$ and $L_2$ mesh norms of the error $E=u_h-u$ by the formulas
\begin{equation*}
\left\|E \right\|_C =\mathop{\max}\limits_{i}\Vert{u_{hi}^{N_t} - u_i^{N_t}} \Vert, \; \left\|E \right\|_{L_2}=\sqrt{\sum\limits_{i = 0}^N {l_i \left({u_{hi}^{N_t} - u_i^{N_t}}\right)^2}}.
\end{equation*}
The rate of convergence (RC) is calculated using double mesh principle
\begin{equation*}
RC=\log_{2}(E^{N}/E^{2N}),\;\; E^{N}=\|u_h^{N}-u^N\|,
\end{equation*}
where $\|.\|$ is the mesh $C$-norm or $L_{2}$-norm, $u^N$ and $u_h^{N}$ are respectively the exact solution and the numerical solution, computed at the mesh with $N$ subintervals. We choose the weight parameter with respect to the time variable $\theta=0.5$.

In Table \ref{t1} we show the convergence and the accuracy of the constructed scheme, where we numerically solve the model problem with the known exact solution $u(x,t)=\exp(x-t)$ and initial data $u_0(x)=\exp(x)$. We select this function because it's feature is similar to that of the exact solution to the call option problem. The results, corresponding to problems $TP1$ and $TP3$ with $\Delta t_m=0.001, m=0,\dots,N_t-1$, are listed in Table \ref{t1}.

\begin{table}[h]
\caption{}%
\label{t1}
\begin{tabular*}{\textwidth}{@{\extracolsep{\fill}}rcccccccccc}
\hline\\[-0.1in]
&&\multicolumn{4}{c}{$TP1$}
&&\multicolumn{4}{c}{$TP3$}
\\[0.05in]\cline{3-6}\cline{8-11}\\[-0.1in]
\multicolumn {1}{c}{$N$}
&&\multicolumn {1}{c}{$E^N_\infty$}& $RC$& $E^N_2$& $RC$
&&\multicolumn {1}{c}{$E^N_\infty$}& $RC$& $E^N_2$& $RC$
\\[0.03in]
\hline \hline \\[-0.1in]
80 && 3.455e-3 & - & 2.801e-4 & - && 4.805e-3 & - & 3.914e-4 & - \\
160 && 1.729e-3 & 0.998 & 9.914e-5 & 1.498 && 2.405e-3 & 0.998 & 1.385e-4 & 1.498 \\
320 && 8.650e-4 & 0.999 & 3.507e-5 & 1.499 && 1.203e-3 & 0.999 & 4.900e-5 & 1.499 \\
640 && 4.326e-4 & 0.999 & 1.240e-5 & 1.499 && 6.015e-4 & 0.999 & 1.733e-5 & 1.499 \\
\end{tabular*}
\end{table}

Figures \ref{f1}-\ref{f4} illustrate the numerical solution, computed with $\tau=0.0001$ on an uniform mesh $N=320$ for $TP1$ and $TP3$, and the well-known solution of the classical Black-Scholes equation, computed by the financial toolbox of MATLAB, \emph{blsprice(Price, Strike, Rate, Time, Volatility)}. Comparison results for $TP3$, where $N=40$ and $T=1$, are given in Figure \ref{f4}, while the numerical solution is visualized in Figure \ref{f3}.

\begin{figure}[htbp]
\hfill%
\begin{minipage}[b]{0.50\textwidth}
\centering
\includegraphics[width=\textwidth]{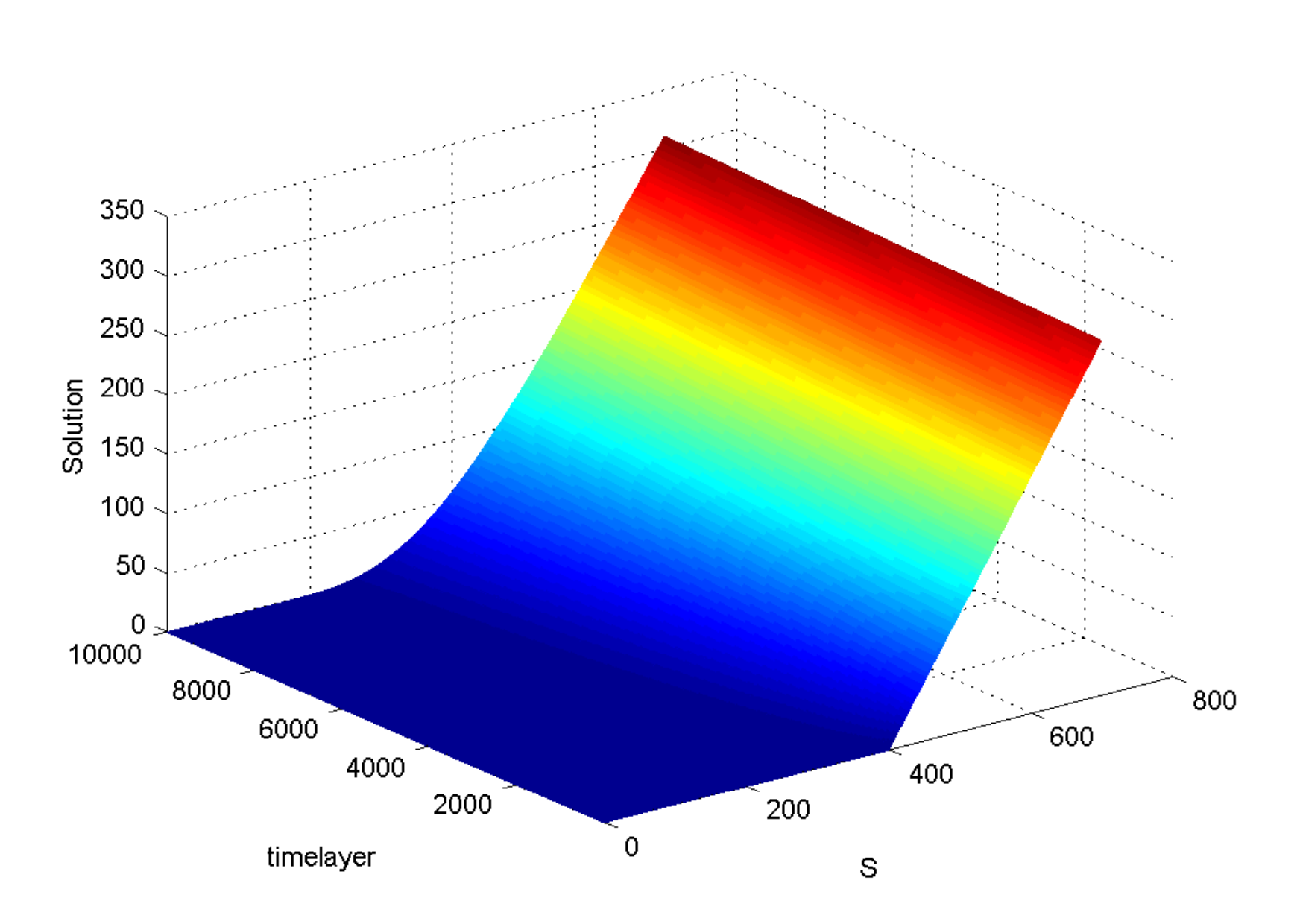}
\caption{Numerical Solution $TP1$}\label{f1}
\end{minipage}%
\hfill%
\begin{minipage}[b]{0.50\textwidth}
\centering
\includegraphics[width=\textwidth]{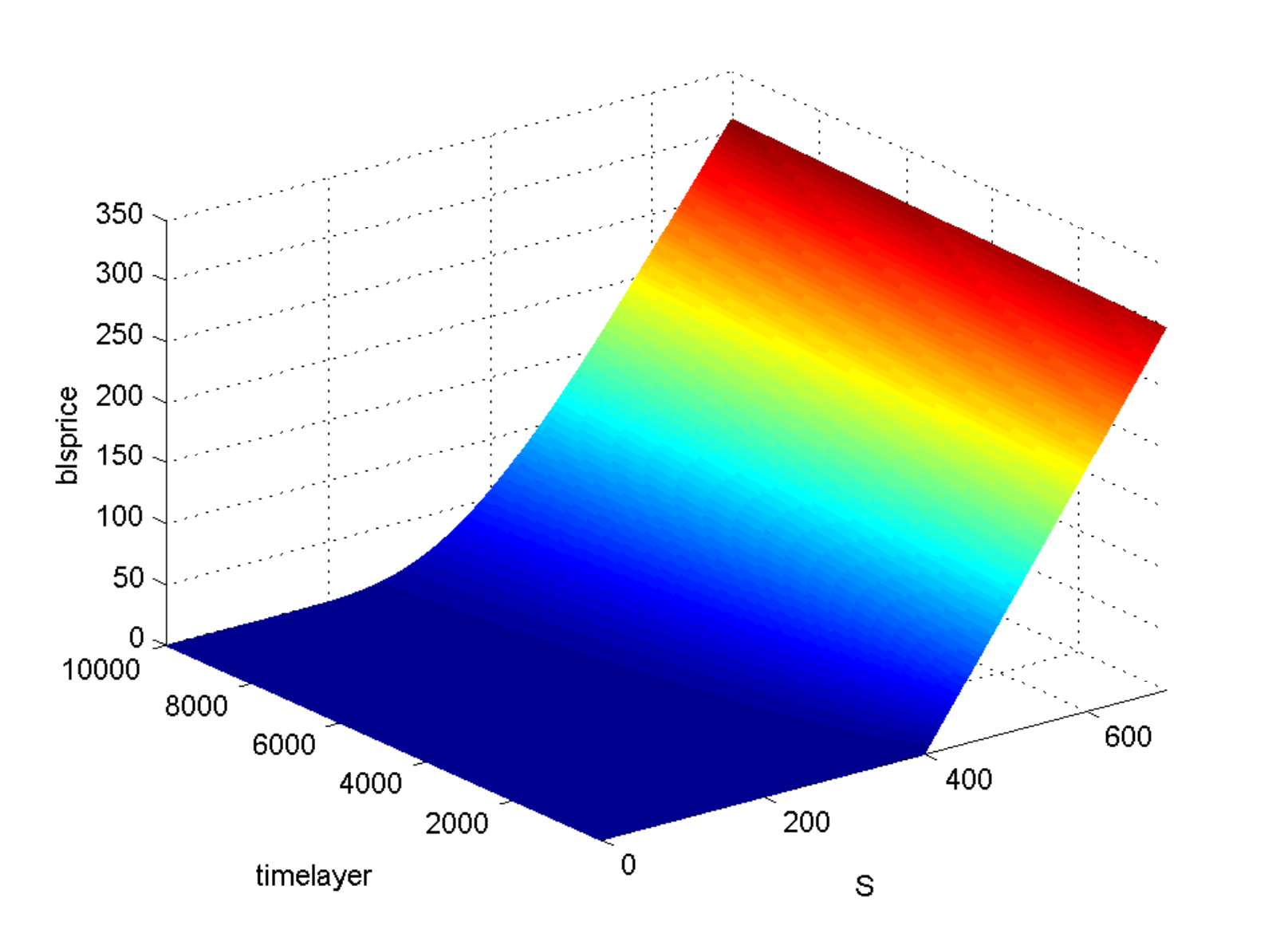}
\caption{Black-Scholes Solution $TP1$}\label{f2}
\end{minipage}\hfill\hbox{}%
\end{figure}

\begin{figure}[htbp]\label{f2}
\hfill%
\begin{minipage}[b]{0.50\textwidth}
\centering
\includegraphics[width=\textwidth]{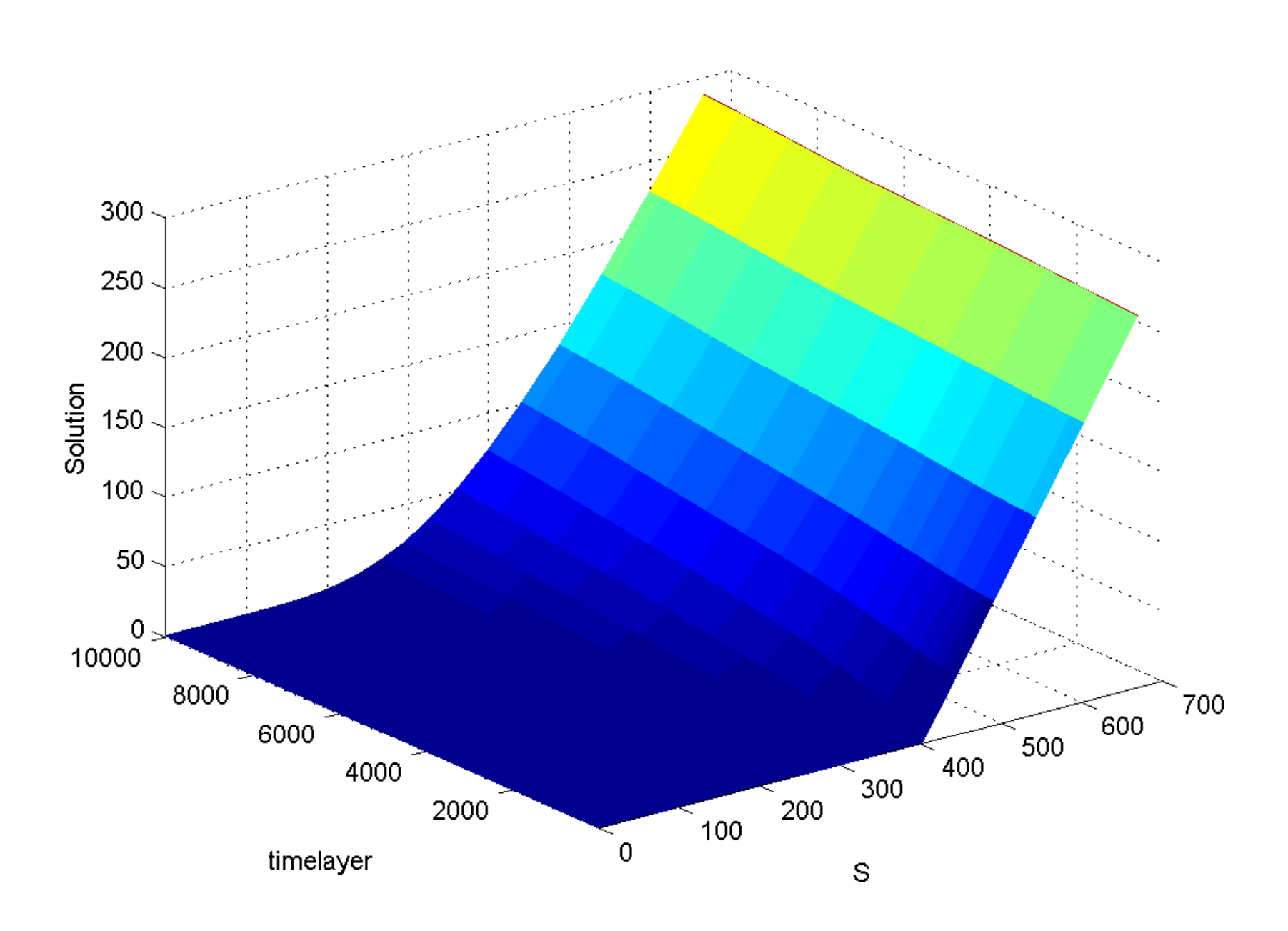}
\caption{Numerical Solution $TP3$}\label{f3}
\end{minipage}%
\hfill%
\begin{minipage}[b]{0.50\textwidth}
\centering
\includegraphics[width=\textwidth]{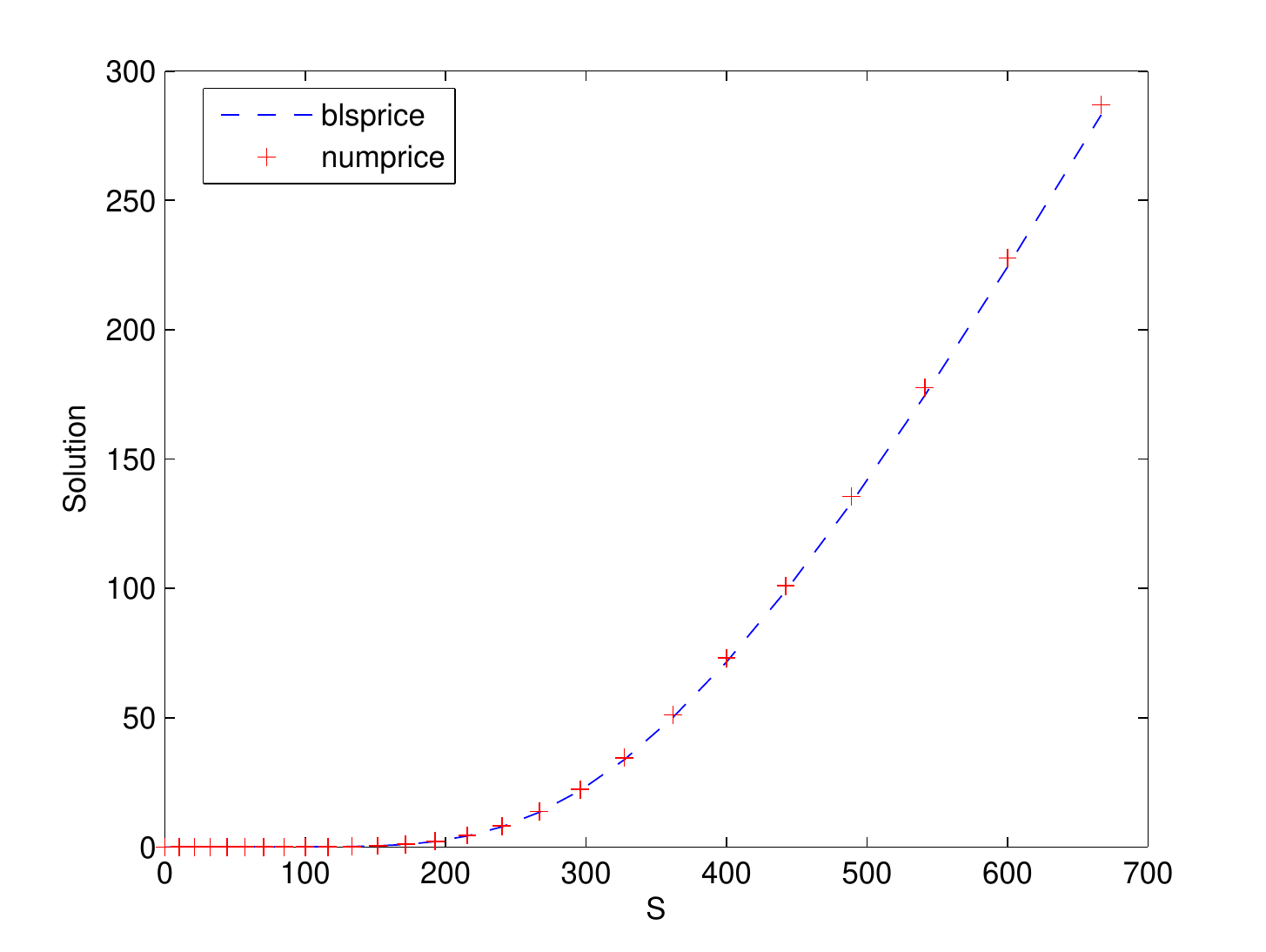}
\caption{Comparison $TP3$}\label{f4}
\end{minipage}\hfill\hbox{}%
\end{figure}

In Table \ref{t2} the results are obtained by computations on a \emph{power-graded} mesh for the same values of the parameters and exact solution. This mesh takes into account the \emph{degeneration at the both ends} of the interval $(0,1)$ and is given by (in the current case $p=2$)
\begin{eqnarray*}
& x_{i+1}=x_{i} + \left( i \left( 2 \sum_{i=0}^{N/2} i^p \right)^{-1/p} \right)^{p}, \; i=1,\dots,N/2, \\
& x_{i+1}=x_{i} + \left( (N+1-i) \left( 2 \sum_{i=0}^{N/2} i^p \right)^{-1/p} \right)^{p}, \; i=N/2+1,\dots,N.
\end{eqnarray*}
The time step $\Delta t_m$ is chosen such that $\Delta t_m = \min_{0 \leq i \leq N} h(i),m=0,\dots,N_t-1$ with $T=0.1$.

\begin{table}[h]
\caption{}%
\label{t2}
\begin{tabular*}{\textwidth}{@{\extracolsep{\fill}}rcccccccccc}
\hline\\[-0.1in]
&&\multicolumn{4}{c}{$TP1$}
&&\multicolumn{4}{c}{$TP3$}
\\[0.05in]\cline{3-6}\cline{8-11}\\[-0.1in]
\multicolumn {1}{c}{$N$}
&&\multicolumn {1}{c}{$E^N_\infty$}& $RC$& $E^N_2$& $RC$
&&\multicolumn {1}{c}{$E^N_\infty$}& $RC$& $E^N_2$& $RC$
\\[0.03in]
\hline \hline \\[-0.1in]
20 && 7.154e-4 & - & 3.648e-4 & - && 6.263e-4 & - & 3.914e-4 & - \\
40 && 1.880e-4 & 1.927 & 9.525e-5 & 1.947 && 1.650e-4 & 1.924 & 8.341e-5 & 2.230 \\
80 && 4.818e-5 & 1.964 & 2.437e-5 & 1.966 && 4.226e-5 & 1.964 & 2.134e-5 & 1.966 \\
160 && 1.220e-5 & 1.982 & 6.167e-6 & 1.982 && 1.970e-5 & 1.982 & 5.401e-6 & 1.982 \\
\end{tabular*}
\end{table}

We now compute the solutions of the original models $TP2$ and $TP3$. As exact solution we use the numerical solution on a very fine uniform grid, i.e. $N=5120$ with $\Delta t_m=0.0001,m=0,\dots,N_t-1$. The results are given in Table \ref{t3}. The numerical solutions of $TP2$ and $TP4$ for $N=640$ are plotted in Figures \ref{f5} and \ref{f6}.

\begin{table}[h]
\caption{}%
\label{t3}
\begin{tabular*}{\textwidth}{@{\extracolsep{\fill}}rcccccccccc}
\hline\\[-0.1in]
&&\multicolumn{4}{c}{$TP2$}
&&\multicolumn{4}{c}{$TP3$}
\\[0.05in]\cline{3-6}\cline{8-11}\\[-0.1in]
\multicolumn {1}{c}{$N$}
&&\multicolumn {1}{c}{$E^N_\infty$}& $RC$& $E^N_2$& $RC$
&&\multicolumn {1}{c}{$E^N_\infty$}& $RC$& $E^N_2$& $RC$
\\[0.03in]
\hline \hline \\[-0.1in]
80 && 2.914e-7 & - & 1.112e-7 & - && 2.681e-3 & - & 2.171e-4 & - \\
160 && 9.914e-8 & 1.555 & 2.841e-8 & 1.968 && 1.321e-3 & 1.021 & 7.476e-5 & 1.538  \\
320 && 5.047e-8 & 0.974 & 7.386e-9 & 1.943 && 6.393e-4 & 1.046 & 2.544e-5 & 1.554 \\
640 && 2.545e-8 & 0.987 & 1.973e-9 & 1.904 && 2.984e-4 & 1.099 & 8.374e-6 & 1.603 \\
1280 && 1.269e-8 & 1.003 & 5.42e-10 & 1.864 && 1.279e-4 & 1.222 & 2.534e-6 & 1.724 \\
\end{tabular*}
\end{table}

\begin{figure}[htbp]
\hfill%
\begin{minipage}[b]{0.50\textwidth}
\centering
\includegraphics[width=\textwidth]{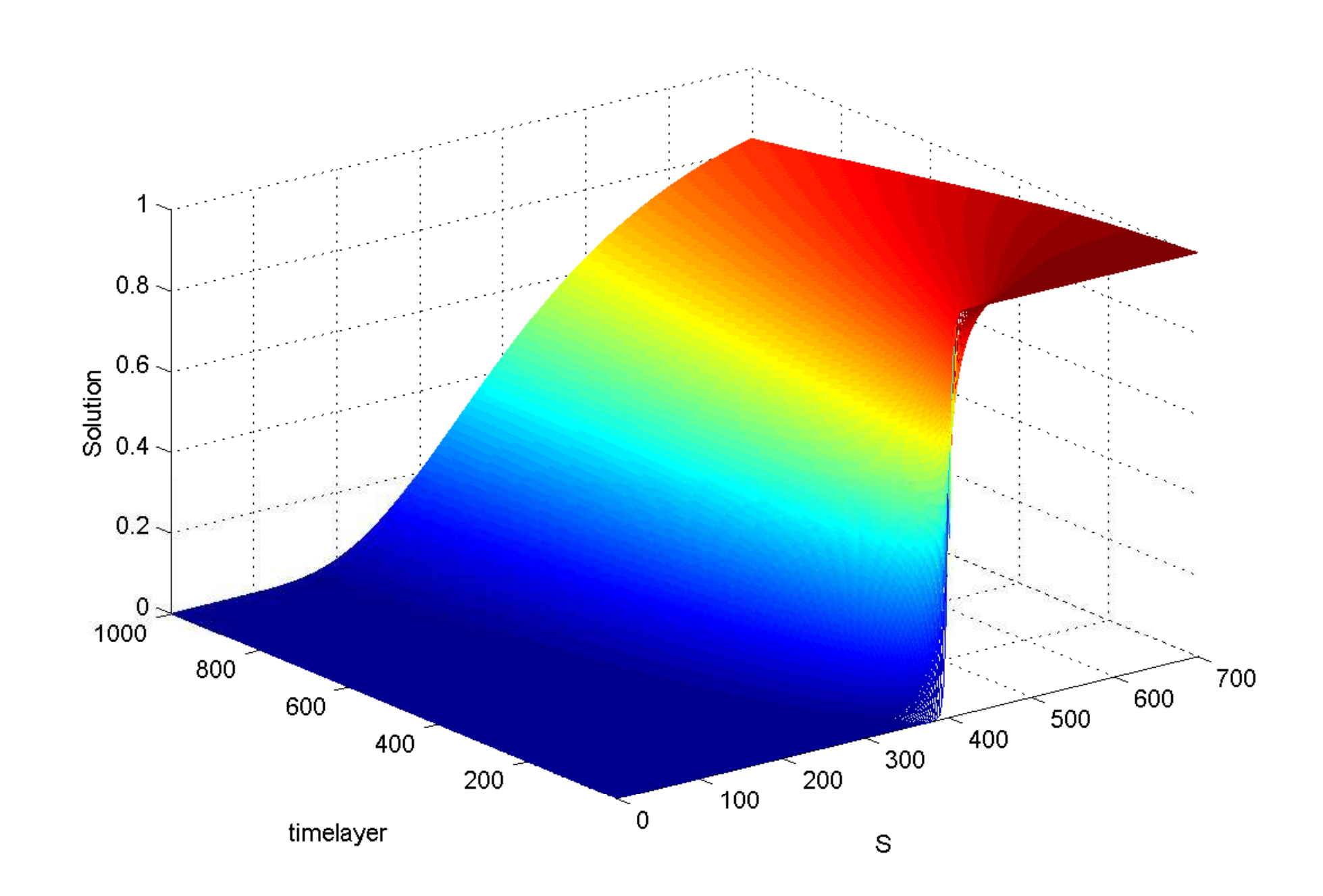}
\caption{Test Problem 2}\label{f5}
\end{minipage}%
\hfill%
\begin{minipage}[b]{0.50\textwidth}
\centering
\includegraphics[width=\textwidth]{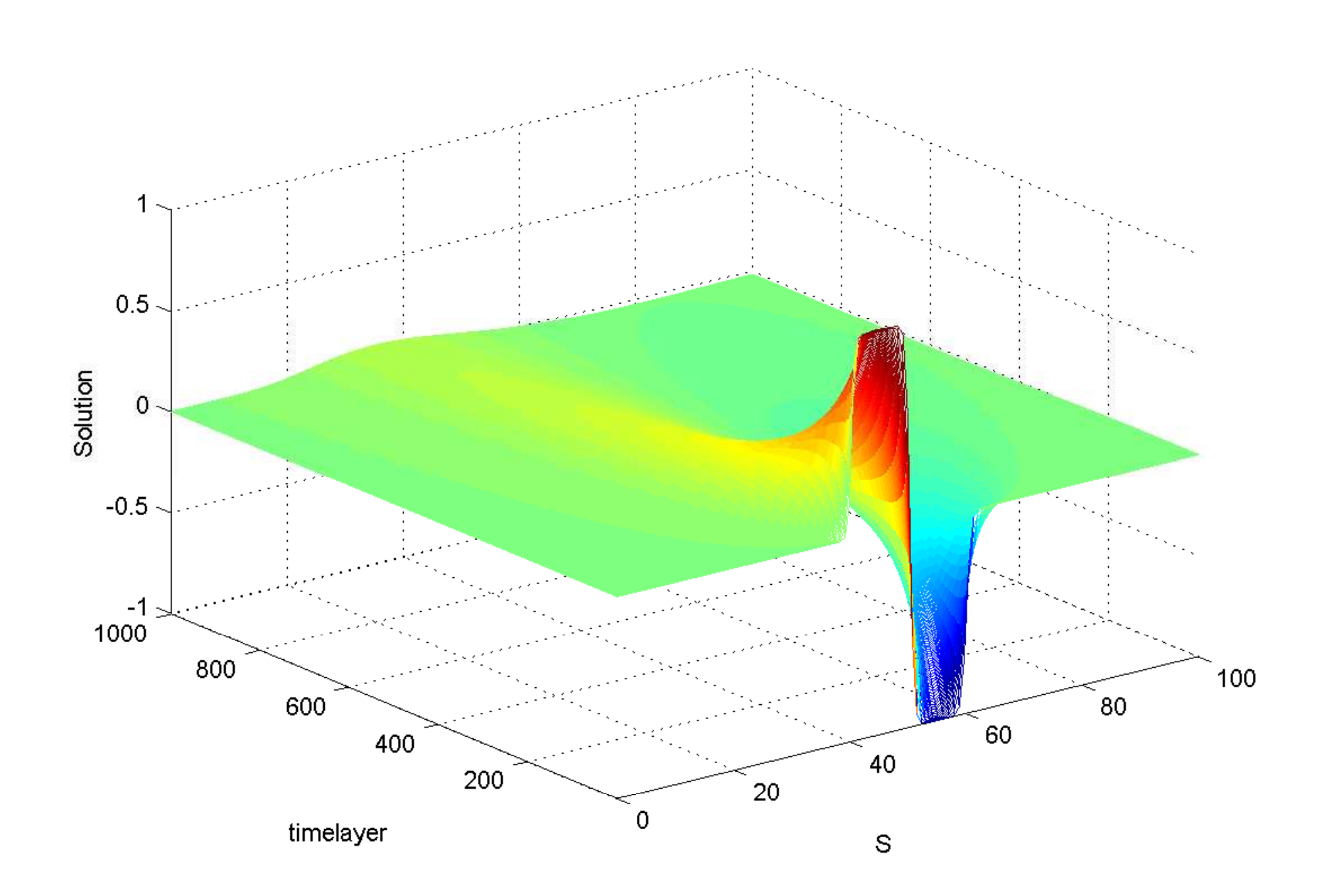}
\caption{Test Problem 4}\label{f6}
\end{minipage}\hfill\hbox{}%
\end{figure}

The convergence of the numerical solution, obtained by the method to the solution of the classical Black-Scholes equation, transformed by \eqref{ZhuTransform}, is given in Table \ref{t4}. The node $x_{N}=1$ is omitted in the calculations since it corresponds to the case $S=\infty$. We use the test parameters in $TP1$ with $d=0$ and $\Delta t_m=0.0001, m=0,\dots,N_t-1$. In the columns 2-5 of Table \ref{t4} we show the overall rate of convergence, while in the last column is given the rate of convergence in the strong norm of the numerical solution in a random node of the mesh, i.e. the one, corresponding to $S=600$. The experiments are performed on an uniform mesh.

\begin{table}[h]
\caption{}%
\label{t4}
\begin{tabular*}{\textwidth}{@{\extracolsep{\fill}}rcccccccc}
\hline\\[-0.1in]
&&\multicolumn{4}{c}{$TP1$}
&&\multicolumn{2}{c}{$S=600$}
\\[0.05in]\cline{3-6}\cline{8-9}\\[-0.1in]
\multicolumn {1}{c}{$N$}
&&\multicolumn {1}{c}{$E^N_\infty$}& $RC$& $E^N_2$& $RC$
&&\multicolumn {1}{c}{$E^N_\infty$}& $RC$
\\[0.03in]
\hline \hline \\[-0.1in]
80 && 3.7473e-4 & - & 6.7765e-5 & - && 1.8848e-5 & - \\
160 && 1.8939e-4 & 0.985 & 2.0388e-5 & 1.733 && 4.7877e-6 & 1.977 \\
320 && 9.5196e-5 & 0.992 & 6.4913e-6 & 1.651 && 1.2016e-6 & 1.994 \\
640 && 4.7722e-5 & 0.996 & 2.1574e-6 & 1.589 && 3.0070e-7 & 1.999 \\
1280 && 2.3892e-5 & 0.998 & 7.3723e-7 & 1.549 && 7.5196e-8 & 1.999 \\
\end{tabular*}
\end{table}

The benchmark of numerical methods in computational finance is the Crank-Nicolson second-order centered space difference scheme (CSDS). It is well known that it produces \emph{spurious oscillations} \cite{AP,Duffy,RS} in the solution and it's spatial derivatives, i.e. $\Delta=\partial V/\partial S$, that are financially unrealistic and are not tolerable. The Figures \ref{osc1} and \ref{osc2} illustrate the problem for the vanilla call option $TP1$, computed on an uniform mesh for $\tau=0.0001$ and parameters $S_{max}=700$, $T=1$, $r=0.1$, $\sigma=0.01$, $d=0$ and $E=400$. We compare the MATLAB function, \emph{blsdelta(Price, Strike, Rate, Time, Volatility)}, with the first derivative of the numerical solution. In the Figures \ref{noosc1} and \ref{noosc2}, generated by the our difference scheme (ODS), no oscillations are observed.

\begin{figure}[htbp]
\hfill%
\begin{minipage}[b]{0.50\textwidth}
\centering
\includegraphics[width=\textwidth]{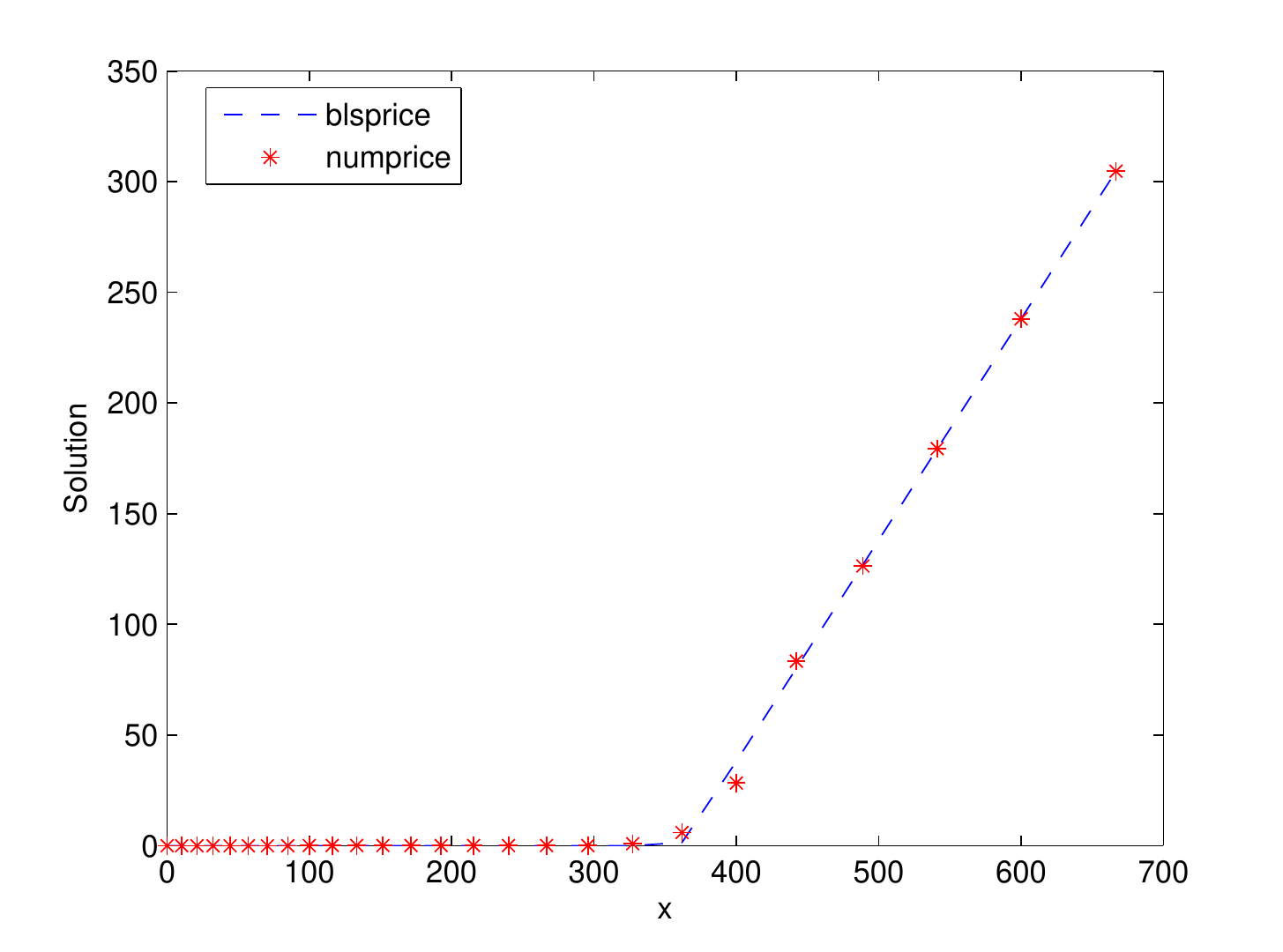}
\caption{Black-Scholes price CSDS}\label{osc1}
\end{minipage}%
\hfill%
\begin{minipage}[b]{0.50\textwidth}
\centering
\includegraphics[width=\textwidth]{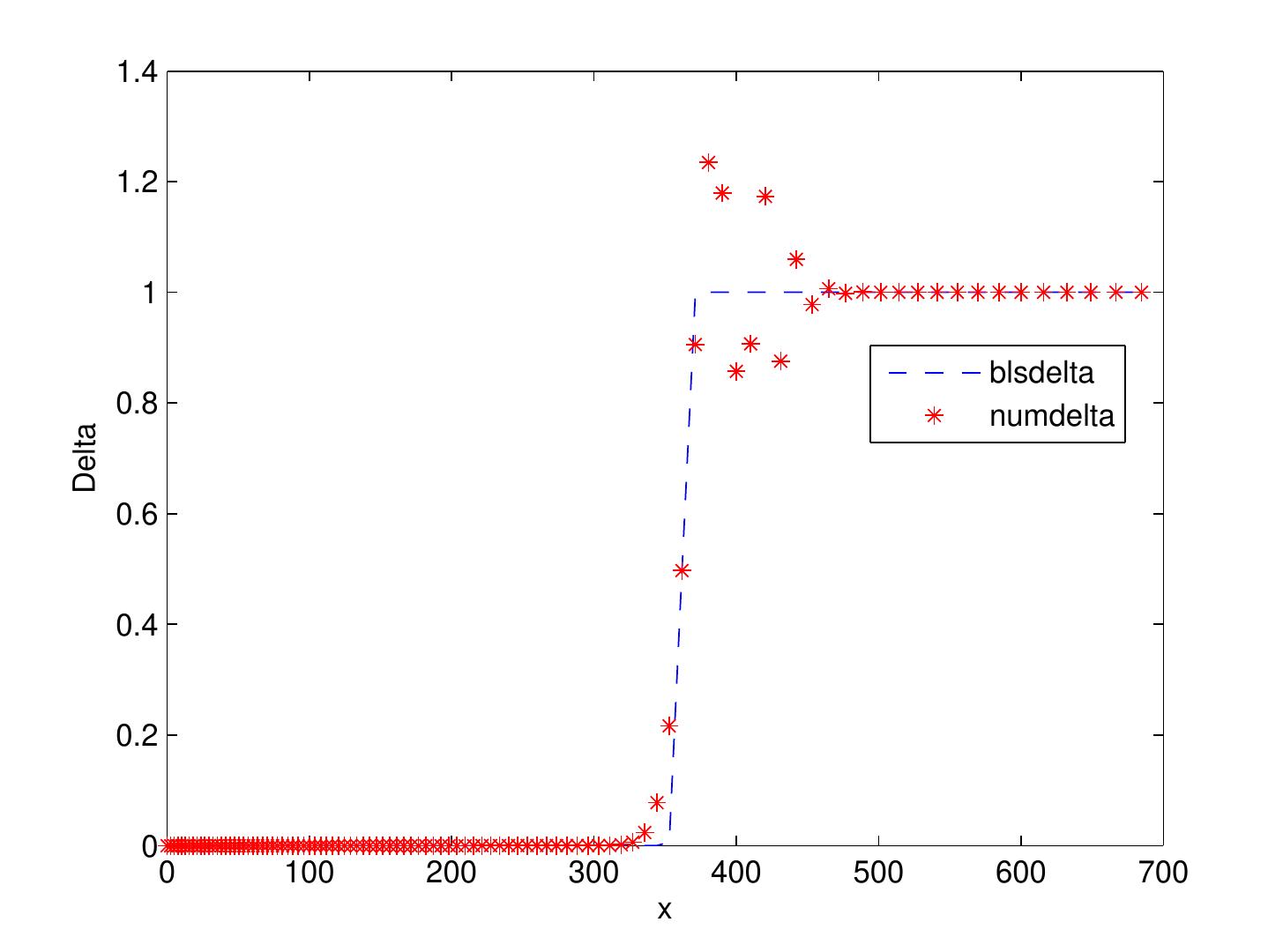}
\caption{Black-Scholes $\Delta$ CSDS}\label{osc2}
\end{minipage}\hfill\hbox{}%
\end{figure}

\begin{figure}[htbp]
\hfill%
\begin{minipage}[b]{0.50\textwidth}
\centering
\includegraphics[width=\textwidth]{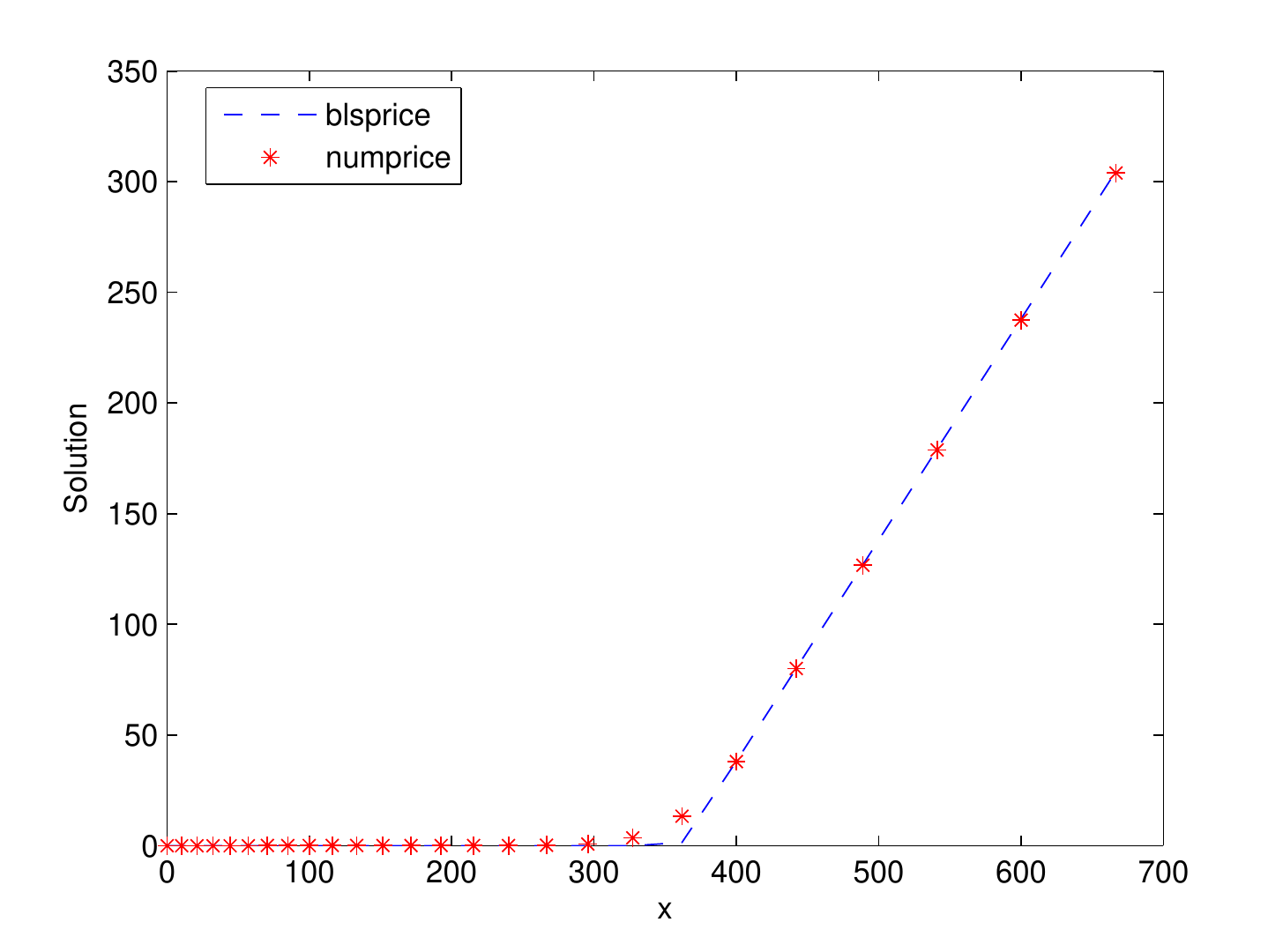}
\caption{Black-Scholes price ODS}\label{noosc1}
\end{minipage}%
\hfill%
\begin{minipage}[b]{0.50\textwidth}
\centering
\includegraphics[width=\textwidth]{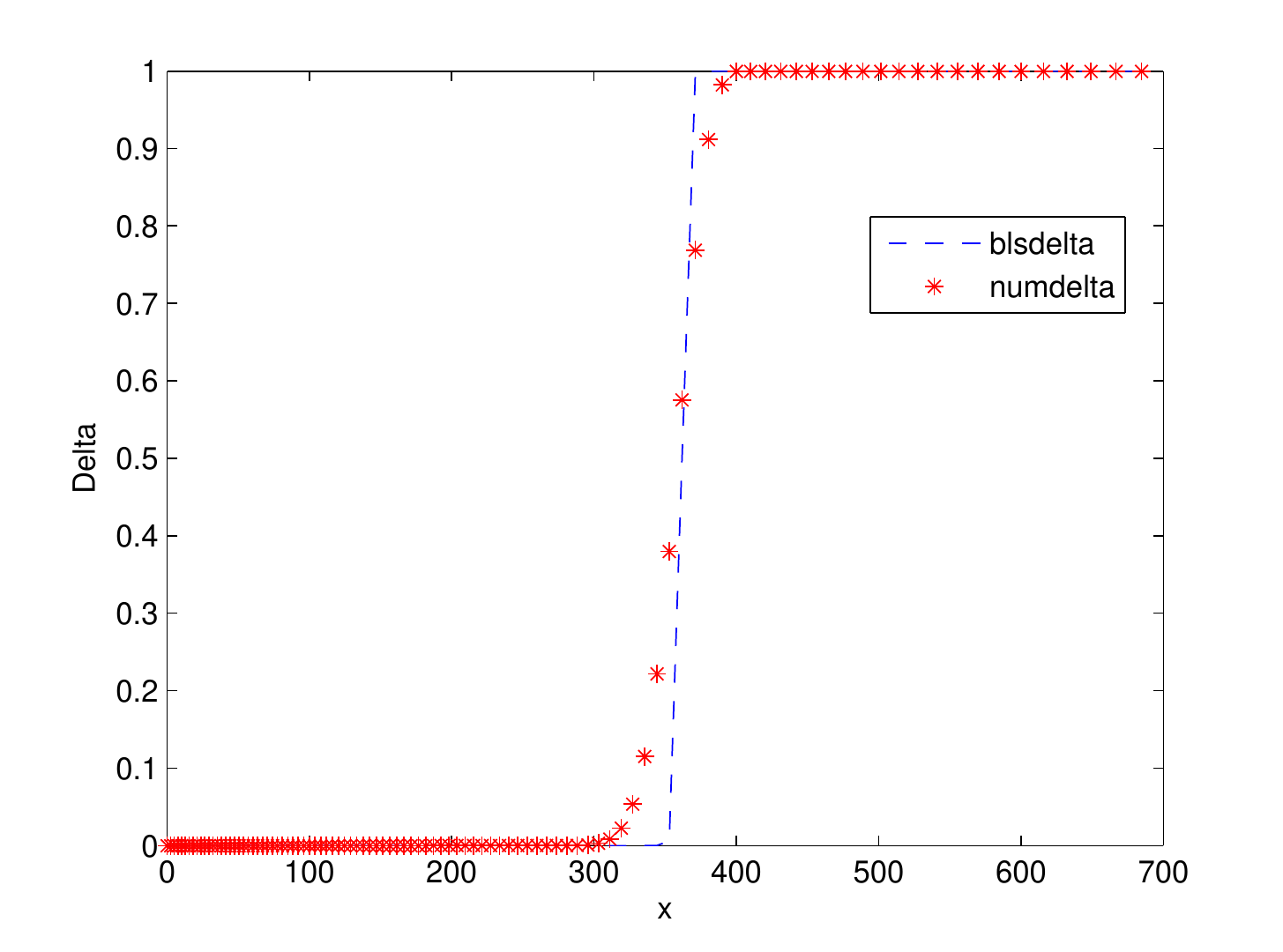}
\caption{Black-Scholes $\Delta$ ODS}\label{noosc2}
\end{minipage}\hfill\hbox{}%
\end{figure}

A realistic situation in financial engineering occurs when the convection and diffusion terms have opposite signs. For example, situations, similar to $r-d<0$, arise in the bond-pricing models, that are also of Black-Scholes type \cite{Duffy,Sevc,Zhu}, where the parameters are interpreted in different context. In Figures \ref{postp2} and \ref{nopostp2} we show the numerical solutions, generated by our difference scheme and by the second-order centered space difference scheme respectively, applied to \eqref{TransfEq}, for initial condition as in $TP2$ with parameters $T=2, \; \sigma=0.1, \; r=0, \; d=0.4x$ and $N=40, \; \tau=0.001$. In Figures \ref{postp3} and \ref{nopostp3} are given the numerical solutions for initial condition as in $TP3$ with parameters $T=2, \; \sigma=0.1, \; r=0, \; d=2x$ and $N=40, \; \tau=0.001$.

\begin{figure}[htbp]
\hfill%
\begin{minipage}[b]{0.50\textwidth}
\centering
\includegraphics[width=\textwidth]{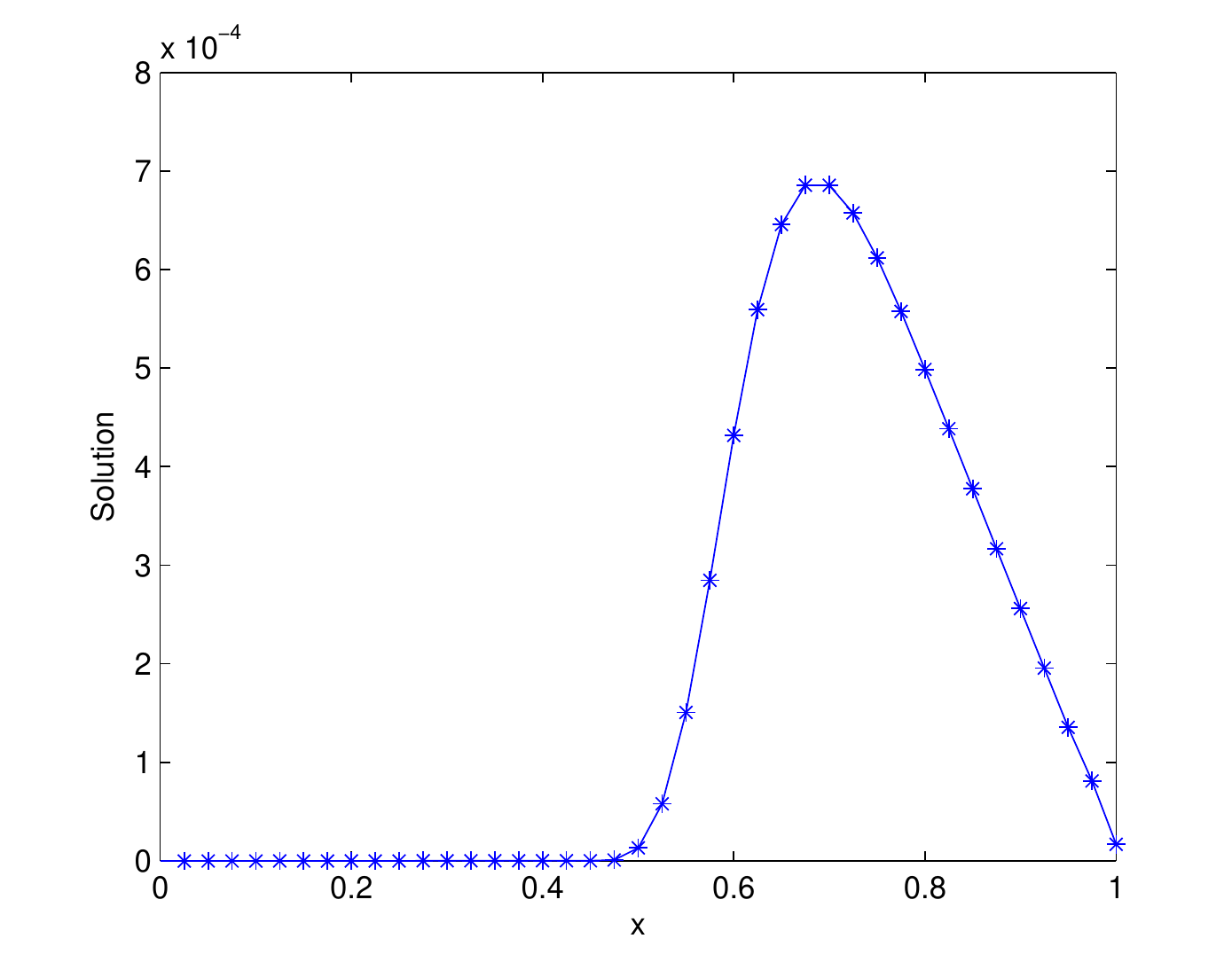}
\caption{Black-Scholes price ODS}\label{postp2}
\end{minipage}%
\hfill%
\begin{minipage}[b]{0.50\textwidth}
\centering
\includegraphics[width=\textwidth]{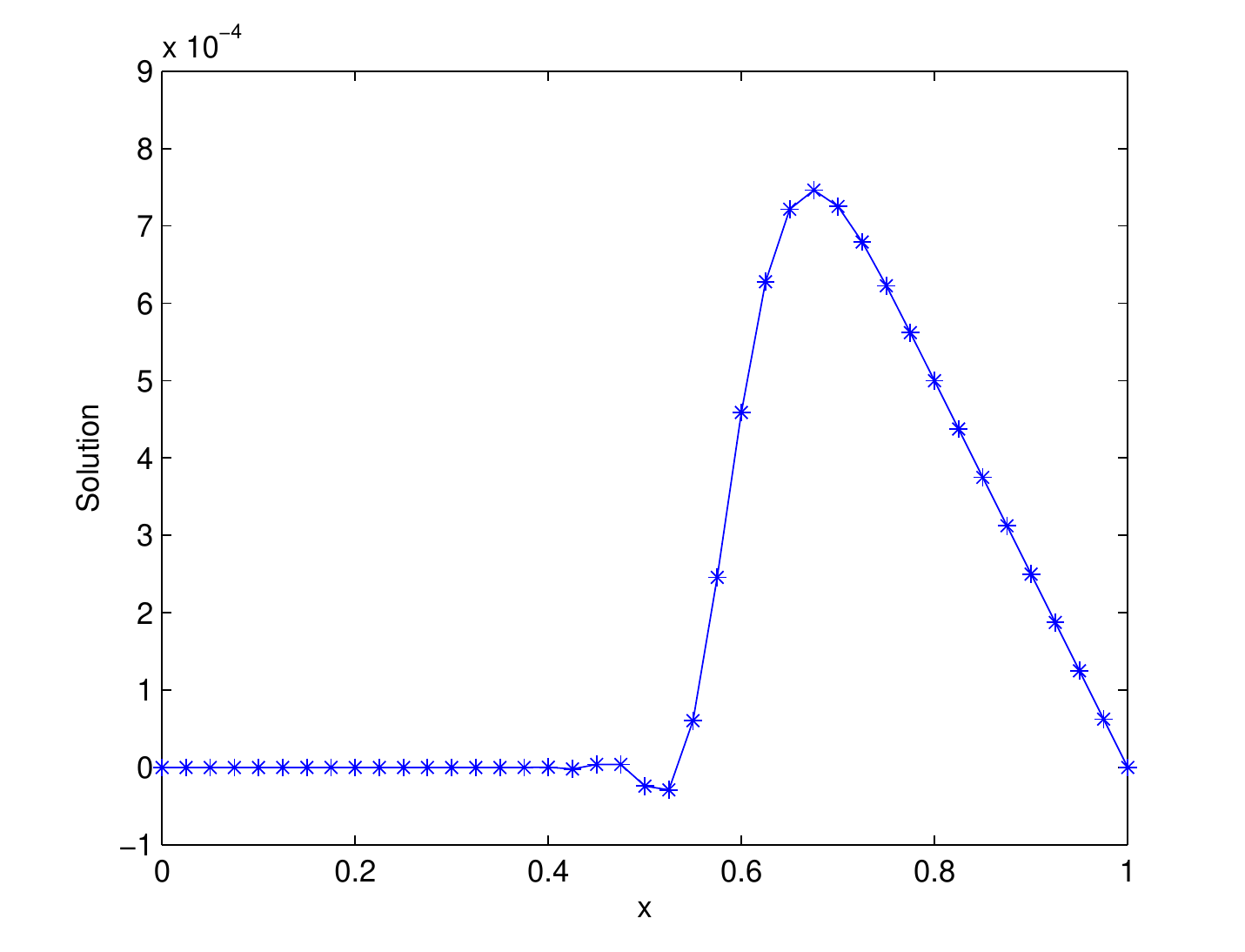}
\caption{Black-Scholes price CSDS}\label{nopostp2}
\end{minipage}\hfill\hbox{}%
\end{figure}

\begin{figure}[htbp]
\hfill%
\begin{minipage}[b]{0.50\textwidth}
\centering
\includegraphics[width=\textwidth]{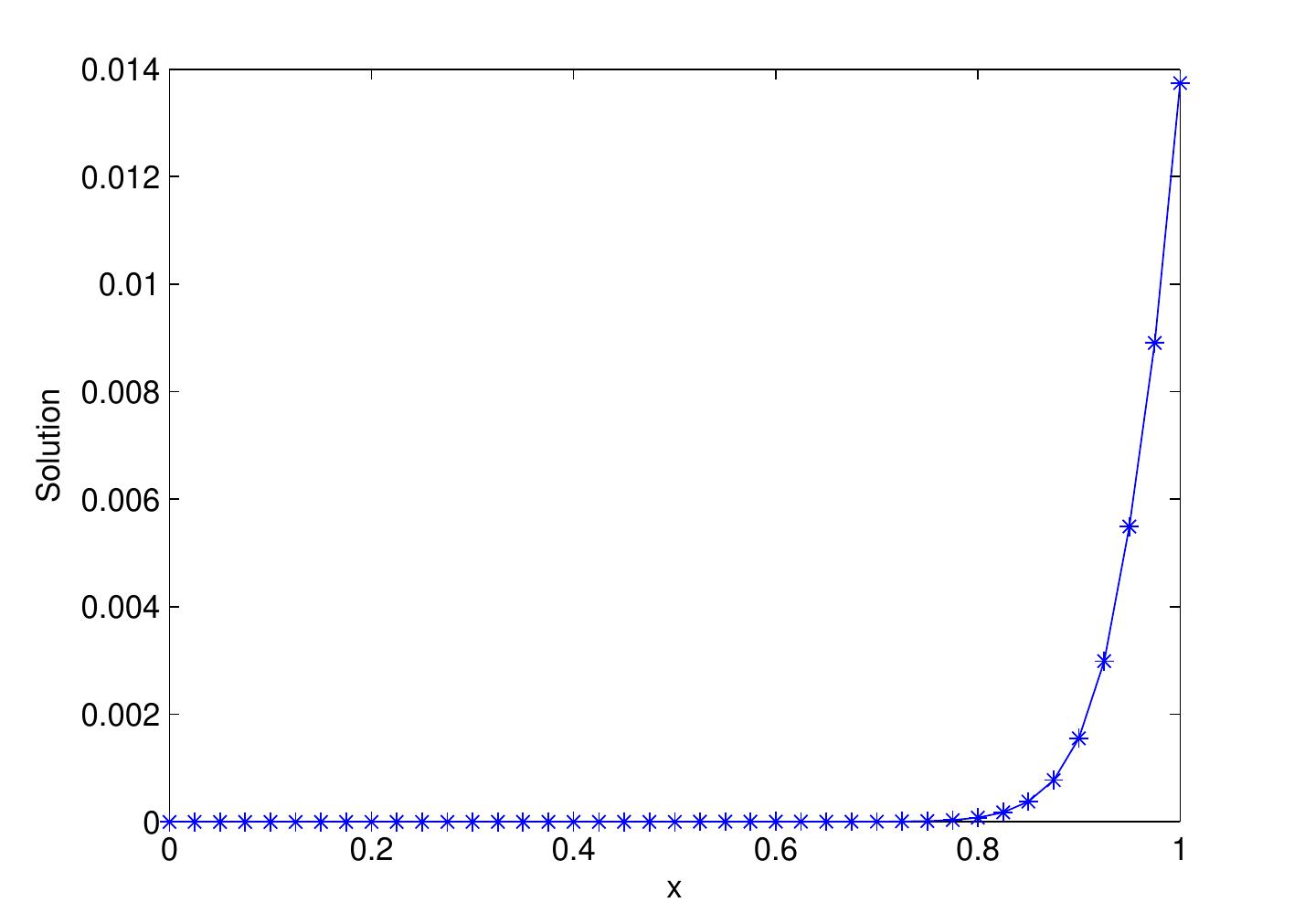}
\caption{Black-Scholes price ODS}\label{postp3}
\end{minipage}%
\hfill%
\begin{minipage}[b]{0.50\textwidth}
\centering
\includegraphics[width=\textwidth]{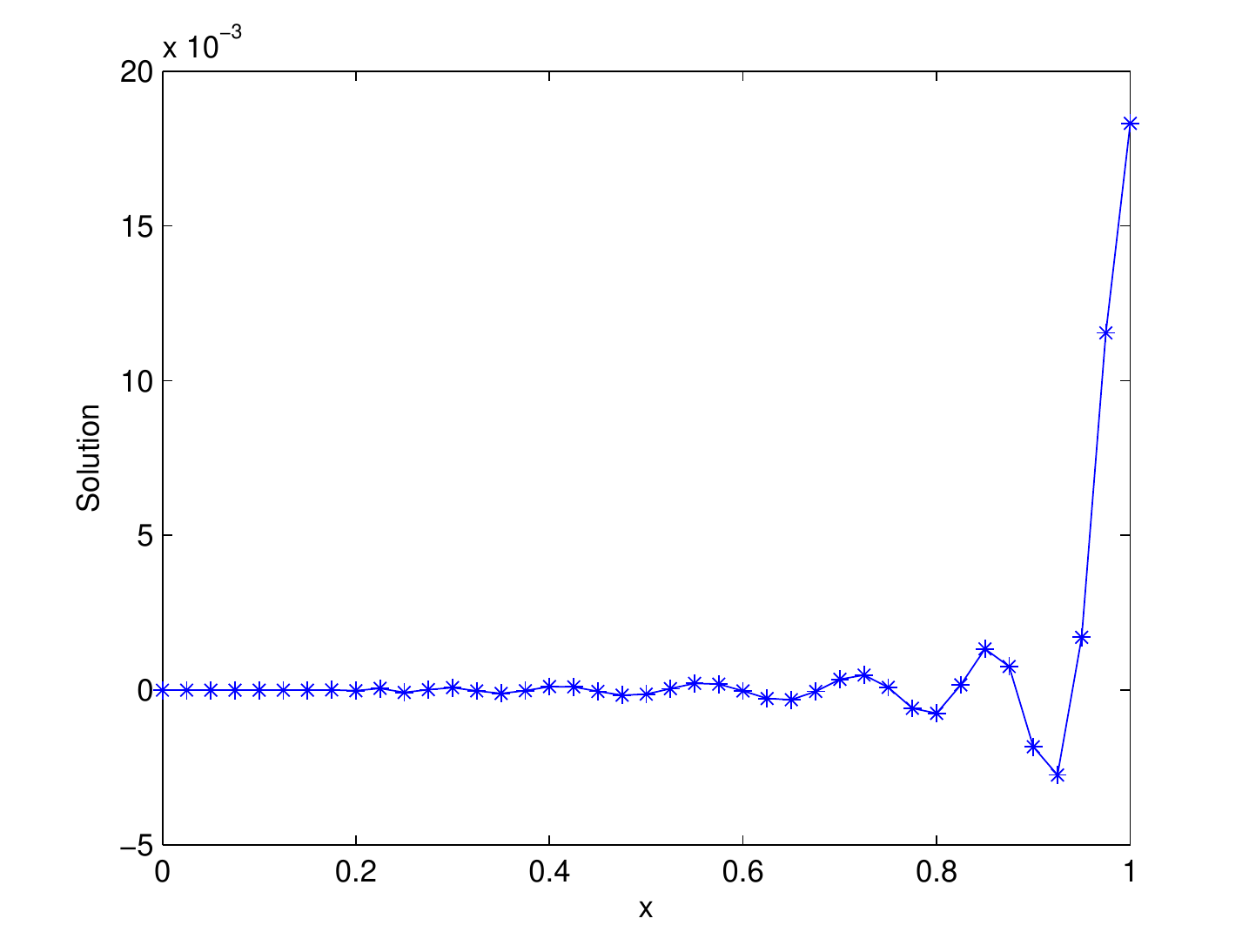}
\caption{Black-Scholes price CSDS}\label{nopostp3}
\end{minipage}\hfill\hbox{}%
\end{figure}

As seen in Figures \ref{postp2}-\ref{nopostp3}, monotonicity and stability are not guaranteed for the centered space difference scheme if the convection and diffusion coefficients are of different signs. Simple calculations show that the discrete maximum principle is violated. We do not observe such problems in the numerical solution, generated by our numerical method.

\section{Conclusion}

In this article we present a fitted FVM for the generalized Black-Scholes equation \eqref{OrigEq}. The method is applicable to more general Black-Scholes models, for example when $\sigma = \sigma(S,t)$ and $r=r(S,t)$. We may also use any interval $(0,l), \; l>0$ (here we took $l=1$ for simplicity) to solve the transformed problem. The main advantage of the developed numerical algorithm is reduction of the computational costs as well as positivity-preserving.

The conducted experiments show first order of convergence of the proposed scheme on a quasi-uniform mesh and second order of convergence on a particular graded mesh. Moreover, they also indicate better stability and unconditional (w.r.t. to the space step) monotonicity in comparison with other known schemes.

In a forthcoming paper we study the stability and the convergence of the proposed finite volume method.

\textbf{Acknowledgement:} The author is supported by the Bulgarian National Fund under Project DID 02/37/09 and by the European Social Fund through the Human Resource Development Operational Programme under contract BG051PO001-3.3.06-0052 (2012/2014).

\end{document}